\documentclass[12pt,reqno]{amsart}
\let\oldll=\ll
\let\oldgg=\gg
\usepackage{mathtools}
\usepackage{amssymb,mathabx,mathrsfs,enumerate}
\usepackage[usenames,dvipsnames]{xcolor}
\usepackage{bm}
\usepackage{textgreek}
\usepackage[margin=1in]{geometry}
\usepackage[pdftitle={Supercritical minimum mean-weight cycles},pdfauthor={Jian Ding, Nike Sun, and David B. Wilson},colorlinks=true,linkcolor=NavyBlue,urlcolor=RoyalBlue,citecolor=PineGreen,bookmarks=true,bookmarksopen=true,bookmarksopenlevel=2,unicode=true]{hyperref}
\usepackage[font=footnotesize]{caption}
\usepackage{microtype}
\usepackage{letltxmacro}
\usepackage{pkg2}
\makeatletter
\let\oldr@@t\r@@t
\def\r@@t#1#2{%
\setbox0=\hbox{$\oldr@@t#1{#2\,}$}\dimen0=\ht0
\advance\dimen0-0.2\ht0
\setbox2=\hbox{\vrule height\ht0 depth -\dimen0}%
{\box0\lower0.4pt\box2}}
\makeatother

\renewcommand{\ll}{\oldll}
\renewcommand{\gg}{\oldgg}

\newcommand{\s}[1]{$\smash{#1}$}
\newcommand{\ee}{e}
\renewcommand{\vec}[1]{\underline{\smash{#1}}}

\DeclareMathOperator{\wgt}{wgt}
\DeclareMathOperator{\len}{len}
\DeclareMathOperator{\range}{range}

\newcommand{\W}{\mathscr{W}}
\newcommand{\LL}{\mathscr{L}}
\newcommand{\CC}{\mathscr{C}}
\newcommand{\Exp}{\operatorname{\textsc{Exp}}}
\newcommand{\mw}{\bar{c}}

\newcommand{\tZ}{\bar{Z}}

\begin{document}

\vspace*{-12pt}
\title{Supercritical minimum mean-weight cycles}
\author[J.~Ding]{Jian Ding$^*$}
\author[N.~Sun]{Nike Sun$^\dagger$}
\author[D.~B.~Wilson]{David B. Wilson}

\thanks{Research supported in part by $^*$NSF grant DMS-1313596 and $^\dagger$NSF MSPRF grant DMS-1401123}

\address{Statistics Department\\University of Chicago\\Chicago, IL 60637}
\address{Microsoft Research\\
One Memorial Drive\\Cambridge, MA 02142 --and--
\newline\indent
Mathematics Department,
Massachusetts Institute of Technology,
Cambridge, MA 02139}
\address{Microsoft Research\\One Microsoft Way\\Redmond, WA 98052}

\begin{abstract}
We study the weight and length of the minimum mean-weight cycle in the stochastic mean-field distance model, i.e., in the complete graph on \s{n} vertices with edges weighted by independent exponential random variables. Mathieu and Wilson showed that the minimum mean-weight cycle exhibits one of two distinct behaviors, according to whether its mean weight is smaller or larger than \s{1/(ne)}; and that both scenarios occur with positive probability in the limit \s{n\to\infty}. If the mean weight is \s{< 1/(ne)}, the length is of constant order. If the mean weight is \s{> 1/(ne)}, it is concentrated just above \s{1/(n\ee)}, and the length diverges with \s{n}. The analysis of Mathieu--Wilson gives a detailed characterization of the subcritical regime, including the (non-degenerate) limiting distributions of the weight and length, but leaves open the supercritical behavior. We determine the asymptotics for the supercritical regime, showing that with high probability, the minimum mean weight is \s{(n\ee)^{-1}[1 + \pi^2/(2 \log^2 n) + O((\log n)^{-3})]}, and the cycle achieving this minimum has length on the order of \s{(\log n)^3}.
\end{abstract}

\maketitle

\section{Introduction} Given a directed or undirected graph with edge weights, a \emph{minimum mean-weight cycle} (\textsc{mmwc}) is any cycle that minimizes the mean weight (ratio of total weight to cycle length) over all cycles in the graph. Finding an \textsc{mmwc} is a fundamental subproblem to a wide variety of algorithms: for example, they have been used (\cite{goldberg-tarjan,radzik-goldberg}, building on \cite{klein}) to give a strongly polynomial-time algorithm for the minimum-cost circulation problem, of which the maximum flow problem is a special case. Other applications of \textsc{mmwc}'s include algorithms for multicommodity flow~\cite{ouorou-mahey} and asymmetric travelling salesman tours~\cite{kleinberg-williamson}; for further applications see the extensive discussion in~\cite{dasdan}.

Several \textsc{mmwc} algorithms are available (see \cite{dasdan,georgiadis} and references therein); and it is of practical interest to understand their runtime in ``average-case'' settings, that is to say, on random inputs. Experimental studies \cite{DG98, DG99, dasdan, georgiadis} suggest that an algorithm due to Young--Tarjan--Orlin \cite{YTO} (based on an improvement of a parametric shortest-path algorithm~\cite{KO81}) has the best runtime in standard random graph ensembles, where it substantially outperforms its worst-case theoretical guarantees.

Motivated by the empirical studies, a natural direction is to understand the typical behavior of the \textsc{mmwc} in random graphs. With this in mind, Mathieu--Wilson \cite{mathieu-wilson} study the \textsc{mmwc} in the stochastic mean-field distance model: the complete graph, or complete digraph, with i.i.d.\ random edge weights. This is a canonical setting for the study of combinatorial optimization problems, both in the mathematics and physics literature:  other examples include minimum spanning tree \cite{MR770868,MR1054012}, shortest path \cite{MR770869,MR1723648,MR2309622,MR2433939}, traveling salesman \cite{mezard-parisi-tsp,MR2104159,MR2600434}, assignment \cite{mezard-parisi-replicas-opt,MR1839499,MR2036492,MR2178256}, spanners \cite{MR2551020}, and Steiner tree \cite{MR2071332,MR2927630}. The analysis of \cite{mathieu-wilson} uncovers an unusual dichotomy for the \textsc{mmwc} in the stochastic mean-field distance model, and characterizes the subcritical regime. In this work we complement their analysis by characterizing the supercritical regime.

\subsection{Main result} Consider the stochastic mean-field distance model where each edge is independently weighted by an exponential random variable of unit rate.  In this setting we study the minimum mean-weight cycle, which is unique with probability one. Denote its mean weight by \s{\W_n}, and its length by \s{\LL_n}, so the cycle has total weight \s{\LL_n\W_n}. This model exhibits an unusual phase transition \cite{mathieu-wilson} at \s{n\W_n=1/e}: In the \emph{subcritical\/} regime \s{n\W_n<1/e}, \s{n\W_n} is non-concentrated, and \s{\LL_n} stays bounded; moreover the limiting distributions of \s{n\W_n} and \s{\LL_n} are precisely characterized. In the \emph{supercritical\/} regime \s{n\W_n>1/e}, \s{n\W_n} is very concentrated with \s{n\W_n-1/e\to0} in probability, while \s{\LL_n} diverges with \s{n}, at least on the order of \s{(\log n)^2\log\log n}. The conclusions of \cite{mathieu-wilson} are much less precise in the supercritical case, leaving open the asymptotic order of \s{n\W_n-1/e} and \s{\LL_n}.

In this paper we characterize the asymptotics of \s{\W_n} and \s{\LL_n} in the supercritical regime.
In particular, our results confirm simulations done by Uri Zwick and the last author
which suggested that \s{\LL_n\asymp(\log n)^3}. Our main theorem is as follows:

\begin{theorem}\label{t:main}
In the complete graph or complete digraph with i.i.d.\ unit-rate exponential edge weights, let \s{\W_n} and \s{\LL_n} be the mean weight and length of the minimum mean-weight cycle, and write 
\beq
c_\star(n) \equiv e^{-1}[1+\pi^2/[2(\log n)^2]]\,.
\eeq
For all \s{\ep>0} there exists a constant \s{C=C(\ep)>0} such that
	\beq\label{e:main.len.wt}
	\liminf_{n\to\infty}
	\P\left( \hspace{-4pt} \left.
	\begin{array}{c}
	|n\W_n-c_\star(n)|(\log n)^3 \le C\\
	\text{and }
	1/C \le \LL_n/(\log n)^3\le C \\
	\end{array}
	\,\right|\, n\W_n >\ee^{-1}
	\hspace{-4pt} \right) \ge 1-\ep.
	\eeq
These bounds are optimal in the sense that
for any interval \s{I} with length \s{|I|\leq 1/C},
	\beq\label{e:main.nondegenerate}
	\limsup_{n\to\infty}
	\P\left(\hspace{-4pt}
	\left.
	\begin{array}{c}
	[n\W_n-c_\star(n)](\log n)^3
	\in I\\
	\text{or }\LL_n/(\log n)^3 \in I
	\end{array}
	\,\right|\, n\W_n >\ee^{-1}
	\hspace{-4pt}\right) \le\ep.
	\eeq
\end{theorem}

\noindent
Theorem~\ref{t:main} implies that in the supercritical regime \s{n\W_n>1/e}, the random variables \s{\log[\LL_n/(\log n)^3]} and \s{(\log n)^3[n\W_n-c_\star]} are tight, with non-degenerate distributions.

\subsection{Proof ideas}
In the remainder of this introductory section we highlight some of the main new ideas in our proof, which allow us to overcome obstacles in the analysis of \cite{mathieu-wilson}. We remark that very similar obstacles arose in an analysis \cite{ding} for a related model, \emph{percolation of averages}, which asks for the longest path with mean weight below a parameter~\s{\lm}. While our work is carried out in the context of the \textsc{mmwc}, we expect our methods may be applied to improve results in \cite{ding}.

In order to show that the supercritical \textsc{mmwc} has mean weight \s{\W_n} in some interval \s{[a,b]} and satisfies some property \textsc{p}, one must show (i) there are no cycles with mean weight in \s{[1/(ne),a]}, (ii) there are no cycles violating \textsc{p} in \s{[a,b]}, and (iii) there is at least one cycle (satisfying~\textsc{p}) in \s{[a,b]}. In particular, (ii) is usually done by first moment arguments, and is clearly easier for less restrictive properties \textsc{p}. On the other hand, a natural approach for (iii) is the second moment method, for which it is often advantageous to make \textsc{p} more restrictive.

Indeed, it is demonstrated in \cite{ding,mathieu-wilson} that a straightforward second moment method on the number of cycles fails, due to an excessive contribution from atypically light cycles (or paths) --- conditioned on finding one atypically light cycle \s{\CC}, it is very likely to have a large number of light cycles overlapping with \s{\CC}.  To address this issue, these works developed a notion of \emph{uniformity}: for a \s{k}-cycle with edge weights \s{w_1,\ldots,w_k} summing to \s{k\mw/n}, we say the cycle is \s{A}-uniform if the process \s{X} with increments \s{X_{i+1}-X_i= (nw_i/\mw-1)} has range at most \s{A} --- meaning the cycle has no excessively light subpath. The typical \textsc{mmwc} is uniform with high probability; but on the other hand the count of uniform cycles has small enough variance for the second moment method to go through.

However, this uniformity property is not sufficiently restrictive to yield accurate implications on the asymptotics of \s{\W_n,\LL_n}.
Indeed, it is clear that the precision achievable
for the window \s{[a,b]}
is dictated by the accuracy with which
\textsc{p} captures the typical properties of the \textsc{mmwc}.
In this work, we restrict further to uniform cycles with a \emph{typical profile}. We defer the formal definitions to Section~\ref{s:second.moment}; roughly speaking, we restrict to cycles for which the associated process \s{X} not only has range \s{\le A}, but furthermore has \emph{typical local times} within its range.
The analysis of Sections~\ref{s:rw}~and~\ref{s:atypical} will show that this restricted property is satisfied with high probability. On the other hand, in Section~\ref{s:second.moment}
we show that the restriction captures the remaining variance, so that the number of cycles restricted in this manner is well-concentrated about its mean. Theorem~\ref{t:main} follows as a consequence.

A key technical ingredient in our proof is a collection of precise estimates for \textit{exp-minus-one\/} random walks (that is, a random walk with increments distributed as \s{E-1}, with \s{E} a unit-rate exponential variable) conditioned to have restricted range. While our results are reminiscent of analogous estimates for simple random walk or Brownian motion,
for general random walks under suitable moment assumptions there is no general theory yielding estimates to the level of accuracy needed for our \textsc{mmwc} analysis.  Our estimates for the exp-minus-one walk are derived in Section~\ref{s:rw}. A crucial input to these estimates is a precise characterization of the principal eigenvalue and eigenfunction for the exp-minus-one walk in an interval with absorbing boundaries. This analysis may be of independent interest, and is presented in Section~\ref{s:eig}.

\subsection*{Acknowledgements} Computer experiments on the mean-field stochastic \textsc{mmwc} performed by Uri Zwick with D.B.W.\ provided valuable intuition at an early stage of this project.  J.D.\ and N.S.\ thank the MSR Redmond Theory Group for its hospitality.

\section{Preliminaries}

\subsection*{Notation}
We write \s{f_n \lesssim g_n} (or \s{g_n \gtrsim f_n}) if there exists an absolute constant \s{C>0} such that \s{f_n \leq Cg_n} for all \s{n\in \mathbb N}. We write \s{\asymp} to indicate that \s{\lesssim} and \s{\gtrsim} both hold. For numerous parameters in the paper, we write for example \s{C = C(\Delta, \ep)} to indicate that \s{C>0} is a number depending only on \s{\Delta} and \s{\ep}.

\subsection{Cycles and paths}

For a cycle \s{\CC} we write \s{\len(\CC)} for the length (number of edges) of~\s{\CC}; when \s{\len(\CC)=k} we refer to \s{\CC} as a \s{k}-cycle. We write \s{\wgt(\CC)} for its total weight, and
	\begin{equation} \label{e:c.bar}
	\mw(\CC)\equiv n\wgt(\CC)/\len(\CC)
	\end{equation}
for its mean weight scaled by \s{n}; we say \s{\CC} is \emph{\s{c}-light} if \s{\mw(\CC)\le c}. We apply these terms to paths as well as cycles;
note that a \s{k}-cycle involves \s{k} vertices
while a \s{k}-path involves \s{k+1} vertices.
To treat both undirected and directed
random networks in a fairly unified manner, we will always take cycles and paths to be directed. In our random network, the weight of any given \s{k}-cycle or \s{k}-path is distributed as the sum of \s{k} independent unit-rate exponential random variables:
that is to say, a gamma random variable with shape parameter \s{k}, with probability density
	\begin{equation}\label{e:gamma.density}
	f_k(x) = \f{e^{-x} x^{k-1} }{ (k-1)!},
	\quad x\ge0.
	\end{equation}
We abbreviate this distribution as \s{\text{Gam}(k)}.
For the sake of review, we repeat the following calculation from \cite{mathieu-wilson}:

\begin{lemma}\label{l:basic.first.moment}
Let \s{Z^k_c} count all \s{c}-light \s{k}-cycles
(\s{k\ge2}), 
and let \s{\tZ^k_c} count all \s{c}-light \s{k}-paths
		(\s{k\ge1}). For all \s{c\lesssim1}
		and \s{0<\delta<1} we have
	\begin{align*}
	\E [Z^k_c-Z^k_{c(1-\delta)}]
	&\asymp \f{(n)_k}{n^k} \f{(ce)^k[1-(1-\delta)^k]}{k^{3/2}}
	\le \f{(ce)^k}{k^{3/2}},\\
	\E \tZ^k_c
	&\asymp \f{(n)_{k+1}}{n^k} \f{(ce)^k}{k^{1/2}}
	\le n\f{(ce)^k}{k^{1/2}}.
	\end{align*}
\begin{proof}
In the complete graph on \s{n} vertices, the number of (directed) \s{k}-cycles (\s{k\ge2}) is \s{(n)_k/k}.
The weight of any given \s{k}-cycle is distributed as a \s{\text{Gam}(k)} random variable, therefore
	\[\E Z^k_c
	= \f{(n)_k}{k}
		\P( \text{Gam}(k) \le ck/n )
	= \f{(n)_k}{k}
		\int_0^{ck/n}
			\f{e^{-x} x^{k-1}}{(k-1)!} \,dx.\]
Since \s{c\lesssim1} we have \s{e^{-x}\asymp1}
uniformly over the range of integration,
therefore
	\begin{align*}
	\E [Z^k_c-Z^k_{c(1-\delta)}]
	&\asymp
		\f{(n)_k}{k}
		\int_{c(1-\delta)k/n}^{ck/n}
		\f{x^{k-1}}{(k-1)!} \,dx
	\asymp
	\f{(n)_k}{k}
		\f{(ck/n)^k[1-(1-\delta)^k]}{k!}\\
	&\asymp
	\f{(n)_k}{n^k}
	\f{(ce)^k[1-(1-\delta)^k]}{k^{3/2}}
	\le \f{(ce)^k[1-(1-\delta)^k]}{k^{3/2}},
	\end{align*}
proving the estimate for \s{c}-light \s{k}-cycles.
The estimate for
\s{c}-light \s{k}-paths
follows by noting that
the number of \s{k}-paths
(\s{k\ge1}) is \s{(n)_{k+1}}.
\end{proof}
\end{lemma}

An easy variation of the preceding calculation shows that in the targeted regime for \s{\W_n}, there can be no cycles of length less than \s{(\log n)^2}:

\begin{lemma}\label{l:log.squared}
Given any \s{c = 1/e[1 + O(1/(\log n)^2)]} and \s{C>0},
the probability that there is a cycle \s{\CC} with
\s{1/e <\mw(\CC) \le c}
and \s{\len(\CC) < C (\log n)^2}
tends to zero in the limit \s{n\to\infty}.

\begin{proof}
The expected number of directed \s{k}-cycles
that are \s{c}-light but not \s{(1/e)}-light is
	\[
	\E (Z^k_c-Z^k_{1/e})
	=
	\f{(n)_k}{k}
	\int_{k/(en)}^{ck/n}  \f{e^{-x}x^{k-1}}{(k-1)!} \,dx
	\lesssim \f{(ce)^k-1}{k^{3/2}}
	\lesssim \f{k/(\log n)^2}{k^{3/2}},
	\]
where the last bound used that
\s{k\lesssim(\log n)^2}.
Summing over \s{2\le k<(\log n)^2}
and applying Markov's inequality proves the claim.
\end{proof}
\end{lemma}

\subsection{Uniformity}

It was previously demonstrated in \cite{ding,mathieu-wilson} that cycles which are
uniform (in a sense defined formally below) have low variance.
For a sequence of weights \s{w_1,\ldots,w_k}
summing to \s{k\mw/n}, let us define its \emph{excedance relative to \s{c}}
(for short, \emph{\s{c}-excedance\/})
to be the quantity
	\beq\label{e:excedance}
	\sum_i (nw_i/c-1) = k(\mw/c-1).
	\eeq

\begin{dfn}\label{d:unif}
Say that a cycle \s{\CC} is \emph{\s{(c,A)}-uniform}
if it has no subpaths with \s{c}-excedance
outside \s{[-A,A]}. We say simply that \s{\CC} is \emph{\s{A}-uniform} if it is \s{(\mw(\CC),A)}-uniform. We apply the same terminology
to paths as well as cycles.
\end{dfn}

\begin{dfn}\label{d:bridge}
Given weights \s{w_1,\ldots,w_k} with mean \s{\mw/n}, the \emph{\s{D}-tilted bridge} is the process
	\[W_j= \sum_{i=1}^j (n w_i-\mw-D/k),
	\quad\text{started from \s{W_0=0}
	and ending at \s{W_k=-D}.}\]
We refer to the \s{D=0} case as the \emph{untilted} bridge: in particular, a cycle
with mean weight \s{\mw/n} is \s{A}-uniform if and only if its untilted bridge has range \s{\le \mw A}.
Sometimes we say \s{k}-bridge to emphasize that the bridge is defined on the time interval \s{[0,k]}.
\end{dfn}

\subsection{Exponential random walks}
\label{ss:intro.exp.rw}
The \emph{exp-minus-one random walk} is the random walk on the real line with step distribution \s{u-1}, where \s{u} is a unit-rate exponential random variable.
Observe that the cycle bridges (Definition~\ref{d:bridge})
can be rescaled to an exp-minus-one version:
to see this, take any positive \s{\mu} and consider the random vector \s{\vec{w}=(w_1,\ldots,w_k)} where the \s{w_i} are independent exponential random variables with mean \s{\mu}. Following the notation \eqref{e:c.bar}, suppose they sum to \s{k\mw/n}. Let \s{\bar{s}} be any positive \s{\mw}-measurable random variable, and consider the process
	\[W_j = \sum_{i=1}^j (n w_i-\bar{s})
	=\bar{s} \sum_{i=1}^j (n w_i/\bar{s}-1),
	\quad 0\le j\le k.\]
Conditioned on \s{\mw}, the vector \s{\vec{w}} is distributed as a uniform sample from the space of all non-negative vectors in \s{\R^k} with sum \s{k\mw/n}, regardless of the value of \s{\mu}.  In particular, taking \s{\mu=\bar{s}/n} shows that
	\begin{equation}\label{e:scale}
	\begin{array}{l}
	\text{the process \s{W/\bar{s}}
	is distributed as an
	exp-minus-one random walk}\\
	\text{started from the origin and
	\emph{conditioned} to be at \s{k(\mw/\bar{s}-1)}
	at time \s{k}.}\end{array}
	\end{equation}
For example, by taking \s{\bar{s}=\mw}, we see that the process with increments \s{(nw_i/\mw-1)} is distributed simply as an exp-minus-one walk conditioned to return to the origin at time \s{k}. We will apply \eqref{e:scale} with \s{\bar{s}\ne\mw} in the proof of Lemma~\ref{l:unif.paths}.

In view of Definition~\ref{d:bridge}, we study the exp-minus-walk with restricted range. In Section~\ref{s:eig} we give a precise computation of the principal eigenvalue \s{\lm_A} of the exp-minus-walk with killing outside \s{[0,A]}. In the limit of large \s{A}
it behaves as
	\begin{equation}\label{e:lambda.A.asymptotic}
	\lm_A=\exp\{ -\pi^2/(2A^2) + O(1/A^3) \}.
	\end{equation}
By comparison, for simple symmetric random walk on the integers killed outside \s{\{1,\dots,A-1\}} (with \s{A} integral), the principal eigenvalue is \s{\lm_A^\textsc{srw}=\cos(\pi/A)}, which behaves in the limit of large \s{A} as \s{\exp\{ -\pi^2/(2A^2) + O(1/A^4) \}} (see e.g.\ \cite{kac}).

\subsection{Proof overview} Having finished our preliminary calculations, we conclude this section by outlining the proof of our main result.

In Section~\ref{s:rw} we prove the necessary estimates for (range-restricted) exp-minus-one walks --- taking as input the principal eigenvalue and eigenfunction for the walk with killing outside the interval \s{[0,A]}, which will be computed in Section~\ref{s:eig}. The most important consequences of this section concern the \emph{exp-minus-one random walk \s{k}-bridge},
by which we mean an exp-minus-one random walk
conditioned to return to the origin in \s{k} steps.
Lemma~\ref{l:bridge} computes the probability \s{R^k_A} for this process to have range \s{\le A}:
	\[
	R^k_A \asymp (\lm_A)^k \f{k^{3/2}}{A^3}
	\quad\quad\text{for all }k\gtrsim A^2.
	\]
Lemma~\ref{l:expected.profile} shows that if we condition this process to have range \s{\le A}, then its local times are comparable with those of the analogously range-restricted Brownian motion.

In Section~\ref{s:atypical} we apply the random walk estimates to rule out supercritical cycles which are atypically light or long. Recall that in Lemma~\ref{l:basic.first.moment} we computed the expectation of the number \s{Z^k_c} of \s{c}-light \s{k}-cycles. Let \s{Z^k_c(A)} count \s{c}-light \s{k}-cycles that are \s{A}-uniform: in Lemma~\ref{l:unif.cycles} we apply the estimate on \s{R^k_A} to prove
	\[
	\E Z^k_c(A)
	= (\E Z^k_c) R^k_A
	\asymp \f{(ce\lm_A)^k}{A^3}
	\quad\quad\text{for all }k\gtrsim A^2.
	\]
If a \s{c}-light cycle fails to be \s{A}-uniform, then we can extract a subpath whose bridge decreases by \s{-A+O(1)}
and is \s{(A+O(1))}-uniform.
We let \s{\tZ^\ell_c(A)} count \s{\ell}-paths of this type:
in Lemma~\ref{l:unif.paths} we apply random walk estimates from Section~\ref{s:rw} to show that
	\[\E \tZ^\ell_c(A)
	\lesssim \f{(ce\lm_A)^\ell}{A^3} \f{n}{e^A}
	\quad\quad
	\text{for all }\ell\ge1.\]
Consequently, if \s{Z^{\ge A^2}_c} counts all
\s{c}-light cycles of length \s{\ge A^2},
we have
	\begin{equation}\label{e:cycles.and.paths.lbd}
	\P\big( Z^{\ge A^2}_c>0 \big)
	\le \sum_{k\ge A^2} \E Z^k_c(A)
		+ \sum_{\ell\ge1}\E \tZ^\ell_c(A)
	\lesssim \sum_{k\ge1} \f{(ce\lm_A)^k}{A^3}
		\Big[1+\f{n}{e^A}\Big].
	\end{equation}
This calculation suggests that we take \s{A=\log n+O(1)}
and rule out values of \s{c} that make \s{ce\lm_A} too small. This leads to the definitions
	\beq\label{e:c.crit}
	A_\circ\equiv\log n\,, \quad\quad\quad
	c_\circ
	\equiv \f{1}{e\lm_{A_\circ}}
		=\frac{1}{\ee} \left[1+\frac{\pi^2}{2(\log n)^2}+O(1/(\log n)^3)\right],
	\eeq
applying \eqref{e:lambda.A.asymptotic}.
It suffices to prove Theorem~\ref{t:main}
with \s{c_\circ} in place of \s{c_\star}.
The lower bound on \s{n\W_n} stated in the theorem
is an easy consequence of Lemma~\ref{l:log.squared}
and \eqref{e:cycles.and.paths.lbd}, and this is the first main consequence of Section~\ref{s:atypical}.

The second main part of Section~\ref{s:atypical} is to rule out cycles that are much longer than \s{(\log n)^3} in the regime \s{|\mw-c_\circ| \lesssim 1/(\log n)^3}. This argument is rather more involved, but it is guided by the same basic principle that the untilted bridge of a cycle either stays in a restricted range, or has sharp decreases over restricted ranges. If a cycle is very long, its bridge must either have many sharp decreases, or else stay in restricted ranges over long intervals. Both events impose a severe probability cost which can be used to rule out the presence of long cycles.

In Section~\ref{s:second.moment} we identify a subcollection of uniform cycles with a ``typical local time profile,'' and show that the number of such cycles is well concentrated about its mean. It follows that these cycles exist (with high probability) whenever their expected number is large. It follows from Lemma~\ref{l:expected.profile} that most of the contribution to \s{\E Z^k_c(A)} comes from cycles having a typical local time profile --- meaning that \s{\E Z^k_c(A)} locates the sharp transition. Theorem~\ref{t:main} then follows in a straightforward manner.

\section{Random walk estimates}\label{s:rw}

In this section we derive the necessary estimates for exp-minus-one walks subject to restrictions on the range of the walk. A few of the estimates require some understanding of the principal eigenvalue \s{\lm_A}, and associated left eigenfunction \s{\vph_A}, for the walk with killing outside the interval \s{[0,A]}. These will be computed in Section~\ref{s:eig}; in the present section we shall require only the facts that
	\begin{equation}\label{e:delta.A-1}
	\lm_A
	= \exp\{ -[1+O(1/A)]\pi^2/(2A^2) \}
	\end{equation}
and that when \s{\vph_A} is normalized to be a probability density on \s{[0,A]},
	\begin{equation}\label{e:delta.A-2}
	\vph_A(x) \asymp \de_A(x)/A^2\,,\\
	\text{ where }
	\de_A(x)
	\equiv \Ind{0\le x\le A}
		[(x+1) \wedge (A-x+1)].
	\end{equation}
We will also make repeated use of a coupling
of Koml\'os--Major--Tusn\'ady
\cite[Theorem~1]{KMT2}:
let \s{(X_i)_{i\ge1}} be i.i.d.\ random variables with zero mean, unit variance, and finite exponential moments. For any \s{\lm>0} there are constants \s{K_1,K_2}, and a coupling of \s{(X_i)_{i\ge1}} to i.i.d.\ standard Gaussian random variables \s{(Y_i)_{i\ge1}}, such that
	\[\P\bigg( \max_{j\le k} \left|
	\sum_{i=1}^j (X_i-Y_i)\right|
	>K_1\log k+x \bigg) \le K_2 e^{-\lm x}.\]

\subsection{Estimates for short time scales} We begin with some estimates for exp-minus-one walks run for \s{k} steps where \s{k} is arbitrary. We will apply these estimates for the case \s{k\lesssim A^2}; in the next subsection we derive better estimates for longer time scales \s{k \gtrsim A^2}.  The following lemma is well known, see e.g., \cite[Lemma~3.3]{pemantle-peres}.

\begin{lemma} \label{l:one.side}
Consider a random walk whose step distribution has zero mean and unit variance. The probability that the walk started from \s{x} survives at least \s{k} steps before going negative is \s{\asymp (x+1)/(k+1)^{1/2}}, uniformly over \s{k\ge0} and \s{0\le x\lesssim k^{1/2}}.
\end{lemma}

Next we review an easy estimate for the probability the random walk stays confined in an interval; this will be substantially refined later for exp-minus-one walks.

\begin{lemma}\label{l:exit.prob}
Consider a random walk whose step distribution has zero mean and unit variance. The probability the walk will survive for at least \s{k} steps  before exiting \s{[0,A]} is \s{\lesssim \exp\{ -\Om(k/A^2) \}} uniformly over all \s{A\gtrsim1}, \s{k\ge0}, \s{X_0\in[0,A]}.

\begin{proof}
Divide the time interval \s{[0,k]} into length-\s{t} subintervals with \s{t\asymp A^2}. It is then enough to note that the probability for the walk \s{X}
to survive over a single subinterval
is bounded away from one, uniformly over \s{A} and over the choice of the starting point \s{x\in[0,A]}.
Indeed, the survival probability is upper bounded by
\s{\P(|X_t-X_0|\le A)},
which is bounded away from one
either by direct calculation with the gamma distribution,
or by applying the central limit theorem for \s{(X_t-t)/\sqrt{t}}.
\end{proof}
\end{lemma}

From now on we restrict our attention to exp-minus-one random walks in intervals \s{[0,A]} with absorbing boundaries. All our estimates hold also for one-minus-exp random walks (up to constant multiplicative error). Recall from \eqref{e:delta.A-2} the definition of \s{\de_A}. We assume from now on that \s{A} is at least of large constant size.

\begin{cor}\label{c:two.side} There is an absolute constant \s{\eta>0} such that for \s{k\le (\eta A)^2}, the probability for the exp-minus-one walk started from \s{x} to survive at least \s{k} steps before exiting \s{[0,A]} is \s{\asymp \de_A(x)/(k+1)^{1/2}}, uniformly over \s{0\le \de_A(x) \lesssim k^{1/2}}.

\begin{proof} As before it suffices to consider \s{k} exceeding any large fixed constant, so that \s{k\asymp k+1}. Assume first that \s{x\le A/2}, so \s{\de_A(x)=x+1}; the case \s{x\ge A/2} follows in a symmetric fashion. From the one-sided bound of Lemma~\ref{l:one.side}, the probability for the walk started from \s{x} to survive at least \s{k} steps before going negative is \s{\asymp (x+1)/k^{1/2}}, which trivially implies the upper bound. For the lower bound it suffices to subtract the probability that the walk exceeds \s{A} before going negative, which has probability \s{\lesssim (x+1)/A \le \eta(x+1)/k^{1/2}}. Taking sufficiently small \s{\eta} gives the lower bound.
\end{proof}
\end{cor}

Let \s{p^k_A(x,y)} denote the probability density for an exp-minus-one walk \s{X} to go from \s{x} to \s{y} in \s{k} time steps without exiting \s{[0,A]}. Let \s{\vph^k_A(x,y)} denote the density at \s{X_k=y} conditioned on survival in \s{[0,A]} for \s{k} steps.
	\[\begin{array}{rl}
	p^k_A(x,y)
	\hspace{-6pt}&\equiv
	\lim_{\ep\downarrow0} \ (2\ep)^{-1}
	\P( |X_k-y|\le\ep; X_t \in[0,A]
	\text{ for all } 0\le t\le k
	\,|\, X_0=x ),\\
	\vph^k_A(x,y)
	\hspace{-6pt}&\equiv
	\lim_{\ep\downarrow0} \
	(2\ep)^{-1}
	\P( |X_k-y|\le\ep \,|\,
	X_t \in[0,A] \text{ for all } 0\le t\le k,
	X_0=x
	).
	\end{array}\]

\begin{lemma}\label{l:short.kernel} It holds uniformly over \s{A,k,x,y} that
	\[p^k_A(x,y)\lesssim
	 \f{\de_A(x) \de_A(y) }{ (k+1)^{3/2}}.\]

\begin{proof}
As before it suffices to consider \s{k} exceeding any large fixed constant, so that \s{k\asymp k+1}. Take exp-minus-one walks \s{X} and \s{Z} started from \s{X_0=x} and \s{Z_0=0} respectively, and take a one-minus-exp walk \s{Y} started from \s{Y_0=y}, with \s{X,Y,Z} mutually independent. Fix also \s{c\in(0,1/2)}. The probability for an exp-minus-one walk to go from \s{x} to \s{[y-\ep,y+\ep]} without exiting \s{[0,A]} is upper bounded by the probability of the intersection of three events:
	\[\begin{array}{rl}
	E_1 \hspace{-6pt}&=
		\set{\text{\s{X} survives at least
		\s{ck} steps
		without exiting \s{[0,A]}}},\\
	E_2 \hspace{-6pt}
		&= \set{\text{\s{Y} survives at least
		\s{ck} steps
		without exiting \s{[0,A]}}},\\
	E_3 \hspace{-6pt}&=
		\set{\text{\s{Z} goes from \s{0}
		to \s{Y_{ck}-X_{ck} + [-\ep,\ep]}
		 in \s{(1-2c)k} steps}}.
	\end{array}
	\]
Now take \s{c=1/3}, so the one-sided bound of Lemma~\ref{l:one.side} gives \s{\P(E_1)\lesssim \de_A(x)/k^{1/2}} and similarly \s{\P(E_2)\lesssim \de_A(y)/k^{1/2}}. The \textsc{clt} then gives \s{\P(E_3 | E_1 \cap E_2, Y_{ck}-X_{ck})\lesssim \ep/k^{1/2}}, uniformly over all choices of \s{Y_{ck}-X_{ck}}. Multiplying these probabilities together proves the bound.
\end{proof}
\end{lemma}

\begin{cor}\label{c:apx.sine}  For \s{k\asymp A^2}, it holds uniformly over \s{x,y} that
	\[p^k_A(x,y)\asymp
		\f{\de_A(x) \de_A(y) }{ A^3}
	\quad\quad\text{and}\quad\quad
	\vph^k_A(x,y) \asymp \f{ \de_A(y)}{A^2}.\]

\begin{proof} Take the same notation as in the proof of Lemma~\ref{l:short.kernel}, and take a constant \s{c\in(0,1/2)} such that \s{A^2\lesssim ck\le (\eta A)^2}. Corollary~\ref{c:two.side} then gives \s{\P(E_1\cap E_2)\asymp \de_A(x)\de_A(y)/k}. Let
	\[
	E_4\equiv \left\{\hspace{-4pt}
	\begin{array}{c}
	\text{\s{Z} goes from \s{0}
		to \s{Y_{ck}-X_{ck} + [-\ep,\ep]}
		in \s{(1-2c)k} steps}\\
	\text{without exiting the interval \s{[-A/3,2A/3]}.}
	\end{array}
	\hspace{-4pt}\right\}
	\]
Conditioned on \s{E_1\cap E_2}, the event \s{E_5} that both \s{X_{ck},Y_{ck}} are in \s{[A/3, 2A/3]} occurs with probability \s{\asymp1}: for \s{\de_A(x),\de_A(y)\gg\log k} this can be deduced directly from the KMT coupling; otherwise one can use the argument from Lemma~\ref{l:one.side} to reduce to the case \s{\de_A(x),\de_A(y)\gg\log k}.
The functional \textsc{clt} gives
	\[\P(E_4 \,|\,
	E_1 \cap E_2 \cap E_5) \asymp
	\ep/(k-2ck)^{1/2}
	\asymp \ep/k^{1/2},\]
so altogether \s{p^k_A(x,y)\gtrsim \de_A(x)\de_A(y)/k^{3/2} \asymp \de_A(x)\de_A(y)/A^3}. Combining with the  upper bound of Lemma~\ref{l:short.kernel} proves \s{p^k_A(x,y) \asymp \de_A(x)\de_A(y)/A^3}. A final application of Corollary~\ref{c:two.side} proves the estimate on the conditional density, \s{\vph^k_A(x,y)\asymp \de_A(y)/A^2}.
\end{proof}
\end{cor}

In preparation for the lemma that follows, observe that the \s{\mathrm{Gam}(t)} density \s{f_t} satisfies
	\beq\label{e:one.over.x.ubd}
	f_t(t+x)
	\le 1/|x+1|
	\text{ for all \s{t\ge1,x\ge -t}}.
	\eeq
For \s{t=1} it is easy to verify that \s{f_1(1+x)= e^{-1-x} \le 1/|x+1|}, uniformly over all \s{x\ge -1}. For \s{t\ge2}, note that \s{f_t(t+x)} is unimodal in \s{x}, with the unique mode at \s{x=-1}. Since the density \s{f_t} integrates to \s{1}, for any \s{x\ge-t} we have
	\[
	1\ge \int_0^{|x+1|} f_t(t-1+u)\,du
	\ge |x+1| f_t(t+x),
	\]
and rearranging proves \eqref{e:one.over.x.ubd}.

\begin{lemma}\label{l:large.jump} It holds uniformly over \s{A,k,x,y} that
	\[p^k_A(x,y)\lesssim\left\{
	\hspace{-4pt}
	\begin{array}{rl}
	(x+1)(A-y+1) / |x-y|^{3} &\text{for } x<y; \\
	(A-x+1)(y+1) / |x-y|^{3} &\text{for } x>y. \\
	\end{array}\right.\]

\begin{proof} Assume \s{0\le x < y \le A}; the case \s{x>y} follows by a symmetric argument. Define \s{x' \equiv x+(y-x)/4}, \s{y' \equiv y-(y-x)/4}; note that we need only consider \s{y-x} exceeding a large constant. Take exp-minus-one walks \s{X} and \s{Z} started from \s{X_0=x} and \s{Z_0=0} respectively, and take a one-minus-exp walk \s{Y} started from \s{Y_0=y}, with \s{X,Y,Z} mutually independent. Define the stopping times \s{\si\equiv\min\set{t\ge0 :X_t\notin[0,x']}} and \s{\tau\equiv\min\set{t\ge0 :Y_t\notin [y',A]}}. The probability for an exp-minus-one random walk to go from \s{x} to \s{[y-\ep,y+\ep]} in \s{k} steps without exiting \s{[0,A]} can be upper bounded by the probability of the intersection of three events,
	\[\begin{array}{l}
	E_1 = \set{ X_\si>x'}\,,\\
	E_2 = \set{ Y_\tau < y'}\,,\\
	E_3=\set{ \si+\tau \le k} \cap
	\set{Z_{k-(\si+\tau)} \in
	 Y_\tau-X_\si + [-\ep,\ep] }\,.
	\end{array}\]
It is clear from Lemma~\ref{l:exit.prob} that the stopping time \s{\sigma} is  integrable, so Wald's identity gives \s{x=\E X_\sigma}. Rearranging gives
	\[
	\P(E_1) =\f{x-a}{b-a}
	\text{ where }
	a\equiv \E[X_\sigma\,|\,X_\sigma<0]
	\text{ and }
	b\equiv \E[X_\sigma\,|\,X_\sigma>x'].
	\]
If \s{X} is either the exp-minus-one or one-minus-exp walk then \s{-1\le a\le0} and \s{x'\le b\le x'+1}, so \s{\P(E_1) \lesssim (x+1)/x'\lesssim (x+1)/(y-x)},
and similarly \s{\P(E_2) \lesssim (A-y+1)/(y-x)}.
Conditioned on \s{E_1\cap E_2}, we have \s{y'-1\le Y_\tau \le y'} a.s., while \s{X_\si-x'} is a standard exponential random variable. For \s{X_\si-x'} large we take the crude bound
	\[\P(  E_3;  X_\si-x' > (y-x)/4 )
	\le
	\max_{t\le k}
	\max_{u\in\R}
		\f{\P(Z_t \in u + [-\ep,\ep])}{e^{(y-x)/4}}
	\lesssim \f{\ep}{ e^{(y-x)/4} }
	\lesssim \f{\ep}{y-x}.
	\]
For smaller values of \s{X_\si-x'} we instead bound
	\[\P( E_3;  X_\si-x' \le (y-x)/4  )
	\le \max_{t\le k}
	\max_{ 0\le u\le 1+(y-x)/4 }
	\P(Z_t \in (y'-x')-u
		+ [-\ep,\ep]).
	\]
which is again \s{\lesssim \ep/(y-x)}
by \eqref{e:one.over.x.ubd}. Multiplying the probabilities together gives the claimed bound.
\end{proof}	
\end{lemma}

\subsection{Estimates for long time scales}
Recall from \eqref{e:delta.A-1} and \eqref{e:delta.A-2} the
definitions of \s{\lm_A,\vph_A,\de_A}.

\begin{lemma}\label{l:long.kernel} For \s{k\gtrsim A^2}, it holds uniformly over \s{x,y} that
	\[p^k_A(x,y)\asymp (\lm_A)^k
		\f{\de_A(x) \de_A(y) }{ A^3}
	\quad\text{and}\quad
	\vph^k_A(x,y) \asymp \f{\de_A(y)}{A^2}.\]

\begin{proof} Take \s{t\le k} with \s{A^2\lesssim t\le (\eta A)^2}, and run the walk for \s{t} time steps: by Corollary~\ref{c:two.side}, the walk survives in \s{[0,A]} up to time \s{t} with probability \s{\asymp \de_A(x)/t^{1/2} \asymp \de_A(x)/A}. By Corollary~\ref{c:apx.sine}, the density at \s{X_t=z} conditioned on survival is \s{\vph^t_A(x,z) \asymp \de_A(z)/A^2}. In particular we can find a large absolute constant \s{C} such that
	\[ C^{-1} \vph^{}_A(z)
	\le \vph^t_A(x,z)
	\le  C \vph^{}_A(z).\]
Therefore, if we start from the distribution \s{\vph^t_A(x,\cdot)} and evolve the walk forward for \s{k-t} steps, the terminal density at \s{X_k=y} (conditioned on survival up to time \s{t}) will be sandwiched between \s{C^{-1} (\lm_A)^{k-t}\vph_A(y)} and \s{C (\lm_A)^{k-t} \vph_A(y)}. Both lower and upper bounds agree up to constant factors with \s{(\lm_A)^k \de_A(y)/A^2}, and multiplying with the probability of survival up to time \s{t} proves the first estimate \s{p^k_A(x,y)\asymp (\lm_A)^k\de_A(x)\de_A(y)/A^3}. Integrating over \s{y} proves that the probability of survival up to time \s{k} is \s{\asymp (\lm_A)^k \de_A(x)/A}, therefore \s{\vph^k_A(x,y) \asymp \de_A(y)/A^2}.
\end{proof}
\end{lemma}

\begin{rmk}\label{r:monotone}
Since \s{p^k_A(A/2,A/2)} is nondecreasing in \s{A}, an immediate consequence of Lemma~\ref{l:long.kernel} is that \s{\lm_A} is nondecreasing in \s{A}.
\end{rmk}

Recall that \s{(\lm_A)^k\asymp1} for \s{k\lesssim A^2}. It therefore holds uniformly over all \s{A,k,x,y} that
	\begin{equation}
	\label{e:all.times.kernel.ubd}
	p^k_A(x,y)
	\lesssim
	\f{(\lm_A)^k \de_A(x)\de_A(y)}{
		(A^2 \wedge (k+1))^{3/2} }\quad
	\begin{array}{l}
	\text{by Lemma~\ref{l:short.kernel}
	for \s{k\lesssim A^2},}\\
	\text{and Lemma~\ref{l:long.kernel}
	for \s{k\gtrsim A^2}.}
	\end{array}
	\end{equation}
For \s{A-x-y\asymp A} we also have the bounds
	\begin{equation}
	\label{e:all.times.large.jump}
	\left.\begin{array}{l}
	p^k_A(x,A-y)\\
	p^k_A(A-x,y)
	\end{array}\hspace{-6pt}\right\}
	\lesssim
	\f{(\lm_A)^k(x+1)(y+1)}{A^3}\quad
	\begin{array}{l}
	\text{by
	Lemma~\ref{l:large.jump}
	for \s{k\lesssim A^2},}\\
	\text{and Lemma~\ref{l:long.kernel}
	for \s{k\gtrsim A^2}.}
	\end{array}
	\end{equation}
In both \eqref{e:all.times.kernel.ubd} and \eqref{e:all.times.large.jump} the \s{\lesssim} can be replaced with \s{\asymp} in the regime \s{k\gtrsim A^2}.

We conclude with our estimates for the exp-minus-one \s{k}-bridge, which we recall is an exp-minus-one walk conditioned to return to the origin at time \s{k}. We first estimate the probability \s{R^k_A} that this process has range at most \s{A},
	\[R^k_A
	\equiv
	\lim_{\ep\downarrow0}
	\f{\P( |\range(X)|\le A;
		|X_k|\le\ep )}{\P( |X_k|\le\ep )},
	\quad
	\text{\s{X} an exp-minus-one walk.}
	\]

\begin{lemma}\label{l:bridge}
For \s{k\gtrsim A^2},
the exp-minus-one \s{k}-bridge
has range at most \s{A} with probability
	\[ R^k_A
	\asymp (\lm_A)^k \f{k^{3/2}}{A^3}.
	\]

\begin{proof} Let \s{X} be an exp-minus-one walk started from \s{X_0=0}. Decompose \s{R^k_A = R^k_{A/4} + R^k_{A,\ge}}, where \s{R^k_{A/4}} is the contribution from the event that the range of \s{X} is smaller than \s{A/4}, while \s{R^k_{A,\ge}} is the contribution from the event that the range is in \s{[A/4,A]}. In order to have range less than \s{A/4} the walk must certainly stay confined within distance \s{A/4} of the origin, so
	\[ R^k_{A/4}
	\le \lim_{\ep\downarrow0}
		\f{\P(
		(X_{0:k})\subseteq [-A/4,A/4];
		|X_k|\le\ep)}
		{\P(|X_k|\le\ep)}
	\asymp k^{1/2} p^k_{A/2}(A/4,A/4)
	\lesssim k^{1/2} (\lm_{A/2})^k / A\]
by \eqref{e:all.times.kernel.ubd}.
For \s{k\gtrsim A^2}
the right-hand side is
\s{\lesssim k^{3/2} (\lm_A)^k/A^3}.

For larger \s{H} we make a more precise calculation.
Let \s{S,T\in[0,k)} denote the times where the walk achieves its minimum and maximum respectively:
the contribution to \s{R^k_{A,\ge}} from \s{S<T} is
	\begin{align*}
	&\asymp
	\int_{A/4}^A
	\lim_{\ep\downarrow 0}
	\f{
	\P(\range(X|_{[0,k]})\in dH;
	S<T;
	|X_k|\le\ep  )
	}{\P(|X_k|\le\ep)}
	\,dH\\
	&\asymp
	k^{1/2}
	\sum_{s,t=0}^{k-1}
	\Ind{s<t}
	\int_{A/4}^A
	p^{t-s}_H(0,H)
	p^{k-t+s}_H(H,0)
	\,dH
	\lesssim
	k^{5/2} \int_0^{3A/4}
		\f{(\lm_{A-\Delta})^k}{(A-\Delta)^6}
		\,d\Delta
	\end{align*}
where in the final step we applied \eqref{e:all.times.large.jump} and made the change of variables \s{H=A-\Delta}. From the expansion \eqref{e:delta.A-1} and the monotonicity of \s{\lambda_A}
(Remark~\ref{r:monotone}) we have
	\beq\label{e:change.in.lm.of.A}
	\f{\lm_{A-\Delta}}{\lm_A}
	\leq \max\Big\{1,\exp\Big\{ \f{-\Theta(\Delta) + O(1)}{A^3} \Big\}\Big\},
	\eeq
so we find that the contribution to \s{R^k_{A,\ge}} from \s{S<T} is
	\[
	\lesssim k^{5/2}\f{(\lm_A)^k}{A^6}
		\int_0^{3A/4}
	\Big(\f{\lm_{A-\Delta}}{\lm_A}\Big)^k \,d\Delta
	\asymp k^{3/2}\f{(\lm_A)^k}{A^3}.
	\]
The contribution to \s{R^k_{A,\ge}} from \s{S>T} has the same value, so the upper bound follows. The lower bound can be obtained in a similar manner, but summing only over pairs \s{s<t} with \s{t-s\gtrsim A^2} and \s{k-t+s\gtrsim A^2} and applying the lower bound from Lemma~\ref{l:long.kernel} (see the comment below \eqref{e:all.times.large.jump}).
\end{proof}
\end{lemma}

We conclude with an estimate on the local time profile for the exp-minus-one bridge, which will be used to control the variance of ``good'' cycles in our second moment argument. Let \s{X} be an exp-minus-one \s{k}-bridge, and let \s{\widetilde{X}} be \s{X} with a constant shift that centers it at \s{A/2}:
	\beq\label{e:X.recentered.around.half.A}
	\widetilde{X}_t\equiv X_t
			+( A-\max X-\min X)/2
	\eeq
We then define the local times of the range-restricted \s{k}-bridge by
	\[\ell^k_A(S)
	\equiv
	\E\bigg[
		\sum_{t=1}^k
		\Ind{\widetilde{X}_t \in S}
	\,\Big|\,|\range(X)|\le A
	\bigg]\quad\quad
		\text{for \s{S\subseteq[0,A]}.}\]

\begin{lemma} \label{l:expected.profile}
For \s{k\gtrsim A^2}, it holds uniformly over subintervals \s{S\subseteq [0,A]} that
the local time of a range-restricted $k$-bridge is bounded by
	\[\ell^k_A(S) \lesssim (1+|S|) \max_{x\in S} \de_A(x)^2.\]

\begin{proof}
Let \s{X} denote an exp-minus-one walk,
and \s{H} its exact range, so that the recentered walk
 \s{\widetilde{X}}
is confined in the interval
\s{I_{A,H} \equiv [(A-H)/2,(A+H)/2]}.
We will compute
	\[
	\ell^k_A(S)
	=
	\lim_{\ep\downarrow0}
	\f{
	\sum_{t=1}^k
	\P( \widetilde{X}_t\in S,
		H\le A;
		|X_k|\le\ep) }
	{\P( H\le A; |X_k|\le\ep)}.
	\]
Lemma~\ref{l:bridge} gives that the denominator is \s{\asymp \ep  (\lm_A)^k k/A^3} (since \s{X} is a walk rather than a bridge, there is a factor of \s{\ep/\sqrt{k}}).  We decompose \s{\ell^k_A = \ell^k_{A,<}+ \ell^k_{A,\ge}} by separating the numerator into the cases \s{H< A/4} and \s{A/4\le H\le A}. For \s{H<A/4}, applying the bound \eqref{e:all.times.kernel.ubd} gives
	\beq\label{e:loc.times.bound.very.small.range}
	\ell^k_{A,<}(S)
	\le \ell^k_{A,<}([0,A])
	\lesssim
	\f{k\, p^k_{A/2}(A/4,A/4)}
		{(\lm_A)^k k/ A^3}
	\lesssim A^2( \lm_{A/2} / \lm_A)^k
	\lesssim A^2
	\lesssim \max_{x\in S}\delta_A(x)^2,
	\eeq
where in the last bound
we used that \s{\widetilde{X}} is confined in
\s{[3A/8,5A/8]}, so in fact \s{\ell^k_{A,<}(S)} is zero
unless \s{\de_A(x)\asymp A} for some \s{x\in S}.

For larger \s{H},
we sum over all possibilities
\s{s,t\in[1,k]}
where the walk achieves its minimum and maximum respectively,
as well as all times \s{u\in[1,k]} where
\s{\widetilde{X}_u\in S}.
For simplicity we consider only the contribution from times \s{s<u<t}; the contribution from other permutations is calculated similarly and will be of the same asymptotic order (up to constants): thus
	\[\ell^k_{A,\ge}(S)
	\lesssim
		\sum_{s,u,t=1}^k
		\f{\Ind{s<u<t}}
			{(\lm_A)^k k/A^3}
		\int_{A/4}^A
		\int_{S \cap I_{A,H}}
		p_H^{u-s}(0,x_H)
		p_H^{t-u}(x_H,H)
		p_H^{k-t+s}(H,0)
		\,dx
		\,dH
	\]
where \s{x_H \equiv x-(A-H)/2}, so that
for \s{x\in I_{A,H}}
we have \s{x_H\in[0,H]} and \s{\de_H(x_H)\le\de_A(x)}.
By symmetry it suffices to consider \s{x_H\le H/2}: then \s{p_H^{u-s}(0,x_H)} can be bounded by \eqref{e:all.times.kernel.ubd}, while \s{p_H^{t-u}(x_H,H)} and \s{p_H^{k-t+s}(H,0)} can be bounded by \eqref{e:all.times.large.jump}.
The total contribution
to \s{\ell^k_{A,\ge}(S)} from times \s{s<u<t}
with \s{u-s\le H^2} is
	\[
	\lesssim
	\f{k^2 }{(\lm_A)^k k/A^3}
		\int_{A/4}^A
		\f{(\lm_H)^k}{H^6}
		\left(\int_S
		\de_A(x)^2 \,dx\right) \,dH
	\lesssim
	|S| \max_{x\in S}\de_A(x)^2,
	\]
while the total contribution
from times \s{s<u<t}
with
\s{u-s\ge H^2} is
	\[
	\lesssim
	\f{k^3 }{(\lm_A)^k k/A^3}
	\int_{A/4}^A \f{(\lm_H)^k}{H^9}
			\left(\int_S
		\de_A(x)^2 \,dx\right) \,dH
	\lesssim
	(k/A^3)|S|\max_{x\in S}\de_A(x)^2.
	\]
Adding these together
and combining with \eqref{e:loc.times.bound.very.small.range}
gives the stated bound.	
\end{proof}
\end{lemma}

\section{Light cycles and long cycles}
\label{s:atypical}

Recall that for a given cycle \s{\CC} we use \s{\mw(\CC) \equiv n\wgt(\CC)/\len(P)} to denote its mean weight scaled by \s{n}. In this section we prove that w.h.p.\ no cycles in the supercritical regime \s{\mw>1/e} have \s{c_\circ-\mw\gg1/(\log n)^3}.  We also show that in the regime \s{|\mw-c_\circ|\lesssim 1/(\log n)^3}, w.h.p.\ there are no cycles of length \s{\gg(\log n)^3}. The formal statement is as follows:

\begin{theorem}\label{t:first}
In the complete graph or complete digraph with i.i.d.\ unit-rate exponential edge weights,
for all \s{\ep>0} there exists a constant \s{C=C(\ep)>0} such that
\begin{equation}\label{e:t.first.too.light}
	\P\left(
	\exists\text{ cycle }\CC:1/e<\mw(\CC)\leq c_\circ-C/(\log n)^3 
	\right) \le\ep.
\end{equation}
For all \s{\ep>0} and all \s{C_1>0} there exists
	a constant \s{C_2=C_2(\ep,C_1)>0} such that
\begin{equation}\label{e:t.first.light.and.too.long}
	\P\left(
	\begin{array}{c}
	\exists\text{ cycle }\CC :1/e< \mw(\CC)
	\le
	c_\circ+C_1/(\log n)^3\\
	\quad\quad\quad\text{and
	\s{\len(\CC)\geq C_2(\log n)^3}}
	\end{array}
	\right) \le\ep.
\end{equation}
\end{theorem}

\subsection{First moment for uniform cycles and paths}\label{sec:uniform-first-moment}

We begin with some first moment estimates for uniform cycles and paths. Let
\begin{equation}\label{e:z.uniform}
	\begin{array}{rl}
	Z^k_c(A)  \equiv
	&\hspace{-6pt}
	\text{number of \s{c}-light
		\s{A}-uniform \s{k}-cycles;}\\
	\tZ^k_c(A)
	\equiv &\hspace{-6pt}
	\text{number of \s{k}-paths
		that, for some \s{A''\in[A-2,A]},}\\
	&\text{have \s{c}-excedance \s{-A''}
		and are \s{(c,A'')}-uniform.}
	\end{array}
\end{equation}
We will also denote
	\[Z^{\ge k}_c(A)
	\equiv \sum_{\ell\ge k}Z^{\ell}_c(A),
	\quad\quad
	Z^{}_c(A)
	\equiv Z^{\ge 2}_c(A),\quad\quad
	\tZ^{}_c(A)
	\equiv \sum_{\ell\ge1}
	\tZ^\ell_c(A).\]
The purpose of defining \s{\tZ_c(A)}
is to bound the number of light cycles which do not contribute
to \s{Z_c(A)}. From Definition~\ref{d:unif}, if a cycle \s{\CC} is \s{c}-light but fails to be \s{A}-uniform, then it has a subpath \s{P} with \s{c}-excedance \s{<-A}. The following lemma shows how to extract
further subpaths from \s{P}
that are both light and uniform:

\begin{lemma}\label{l:many.subpaths}
For \s{2\le A'\le A}, given a path \s{P} with \s{c}-excedance \s{<-A}, one can extract at least \s{1+\lfloor A-A' \rfloor} distinct (though not necessarily disjoint) subpaths \s{p \subseteq P}, each contributing to \s{\tZ_c(A')} as defined in \eqref{e:z.uniform}.

\begin{proof} By scaling we can assume \s{c=1}.  Suppose \s{P} has length $k$ and edge weights \s{w_1,\ldots,w_k}.  By assumption, the process \s{W_j=\sum_{i=1}^j(nw_i-1)} goes from \s{W_0=0} to \s{W_k<-A}.  Let
	\[(x_i,y_i) = -(i,i+A'),
	\quad\text{for }
	0\le i\le \lfloor A-A' \rfloor.\]
Since \s{A'\ge2}, \s{y_i+1<x_i-1} for each \s{i}.  Since \s{W} decreases by at most one at each step, for each \s{i} there is some \s{\eta_i} for which \s{W_{\eta_i} \in (x_i-1,x_i]}, and some \s{\tau_i>\eta_i} for which \s{W_{\tau_i} \in (y_i,y_i+1]}.
Let \s{[\bar\eta_i,\bar\tau_i]} be any minimal subinterval of \s{[0,k]} for which \s{W_{\bar\eta_i} \in (x_i-1,x_i]} and \s{W_{\bar\tau_i} \in (y_i,y_i+1]}, and let \s{p_i} be the subpath of \s{P} corresponding to the interval \s{[\bar\eta_i,\bar\tau_i]}.
Since \s{[\bar\eta_i,\bar\tau_i]} is minimal, the subpath \s{p_i} is \s{(W_{\bar\eta_i}-W_{\bar\tau_i})}-uniform, where \s{A'-2\leq W_{\bar\eta_i}-W_{\bar\tau_i} \leq A'}.  Since the intervals \s{(x_i-1,x_i]} are disjoint, the paths \s{p_i} are distinct.
\end{proof}
\end{lemma}

Applying Lemma~\ref{l:many.subpaths}
with \s{A'=A}, when \s{A\geq 2}, we see that any \s{c}-light cycle either contributes to \s{Z_c(A)},
or has a subpath contributing to \s{\tZ_c(A)}.

\begin{lemma}\label{l:unif.cycles}
For \s{k\gtrsim A^2},
\s{c\lesssim1}, and \s{0<\delta<1},
	\begin{equation}\label{eq:E[Z_c^k(A)]}
	\E [Z^k_c(A)-Z^k_{c(1-\delta)}(A)]
	\asymp
	\f{(n)_k}{n^k} \f{(ce\lm_A )^k
		[1-(1-\delta)^k]}{A^3}.
	\end{equation}
	
\begin{proof}
Recalling Definition~\ref{d:unif}, a \s{k}-cycle with edge weights \s{w_1,\ldots,w_k} summing to \s{k\mw/n} contributes to \s{Z^k_c(A)} if and only if \s{\mw\le c},
and the untilted bridge
with increments \s{(nw_i/\mw-1)}
has range at most \s{A}. Applying \eqref{e:scale} with \s{\bar{s}=\mw}, we see that \s{\E Z^k_c(A)= (\E Z^k_c) R^k_A}, and then \eqref{eq:E[Z_c^k(A)]} follows from Lemmas~\ref{l:basic.first.moment}~and~\ref{l:bridge}.
\end{proof}
\end{lemma}

\begin{lemma}\label{l:unif.paths}
For \s{A\geq2}, \s{c\lesssim1}, and \s{k\ge 1},
	\begin{equation}\label{eq:E[tZ_c^k(A)]}
	\E \tZ^k_c(A)
	\lesssim 
		\f{(ce\lm_A)^k}{A^3} \f{n}{e^A}.
	\end{equation}

\begin{proof}
A \s{k}-path with edge weights \s{w_1,\ldots,w_k} summing to \s{k\mw/n} contributes to \s{\tZ^k_c(A)} if and only if, for some \s{A''\in[A-2,A]}, the process with increments \s{(nw_i/c-1)} goes from \s{0} to \s{-A''} without exiting \s{[-A'',0]}. In particular this implies \s{\mw=c(1-A''/k)}. Making the change of variables \s{A''=A-\Delta}, we have
	\[\E\tZ^k_c(A)
	= (n)_{k+1}
	\int_0^2
	f_k[( k-A+\Delta ) c/n]
	\f{p^k_{A-\Delta}(A-\Delta,0)}
	{f_k( k-A+\Delta )}\,\f{c\,d\Delta}{n}
	\]
where the first factor in the integrand is the probability density of the path weight, and the second factor is the probability that an exp-minus-walk, conditioned to go from \s{A-\Delta} to \s{0} in \s{k} steps, does so without exiting \s{[0,A-\Delta]}. Simplifying gives
	\[\E\tZ^k_c(A)
	\leq (n)_{k+1}\f{c^k}{n^k}
	\int_0^2 e^{k-A+\Delta}
	p^k_{A-\Delta}(A-\Delta,0)
	\,d\Delta
	\lesssim
	\f{n}{e^A}
	\f{(ce\lm_A)^k}{A^3}
	\int_0^2
	\Big(\f{\lm_{A-\Delta}}{\lm_A}\Big)^k \,d\Delta,
	\]
where the last bound follows from \eqref{e:all.times.large.jump}.  The final integrand is $\leq1$, which gives \eqref{eq:E[tZ_c^k(A)]}.
\end{proof}
\end{lemma}

\subsection{Uniform cycles and light cycles}\label{sec:uniform-light}

We now apply our first moment estimates to argue that in the supercritical regime \s{c>1/e}, cycles cannot be too uniform or too light.  Recall \eqref{e:c.crit}
that we set \s{A_\circ\equiv\log n} and \s{c_\circ \equiv 1/(e\lm_{A_\circ})}.
More generally, we define the relation
	\begin{equation}\label{e:gamma}
	c\,e\lm_A = \exp\{-\gm/A^3\}
	\end{equation}
--- where we can think of \s{c} as being determined by \s{A} and \s{\gm}.
A first consequence of Lemma~\ref{l:unif.cycles}
is that cycles of length at least \s{(\log n)^2}
cannot be too uniform in the targeted regime for the mean weight:

\begin{lemma}\label{p:too.unif}
For all \s{\ep,C>0} there exists \s{\Delta=\Delta(\ep,C)>0}
large enough that
	\[
	\P\left(\begin{array}{c}
		\exists\text{ cycle }\CC:
		\mw(\CC)\le c_\circ+C/(\log n)^3,\\
		\text{\s{\len(\CC)\ge A^2},
		and \s{\CC} is \s{A}-uniform}
		\end{array}
		\right)\le\ep
		\quad\quad\text{for }
		A\equiv \log n-\Delta.
	\]

\begin{proof}
Suppose \eqref{e:gamma} holds with \s{\gm>0}.
Using Markov's inequality with Lemma~\ref{l:unif.cycles} and \eqref{e:gamma},
	\[\P(Z^{\ge A^2}_c(A)>0)
	\le\E Z^{\ge A^2}_c(A)
	\lesssim A^{-3}\sum_{k\ge A^2}e^{-k\gm/A^3}
	\lesssim 1/\gm\,,\]
so by taking \s{\gm\asymp 1/\ep} we can ensure \s{\P(Z^{\ge A^2}_c(A)>0)\le\ep}.
Using the expansion \eqref{e:delta.A-1} for \s{\lambda_A},
we see that by choosing \s{A=\log n-\Delta} with \s{\Delta\asymp C+\gamma},
where \s{\gamma} and \s{C} are fixed as \s{n\to\infty}, we can ensure \s{c=1/(e\lambda_A) \exp(-\gamma/A^3) \geq c_\circ + C/(\log n)^3}.
\end{proof}
\end{lemma}

We next show how Lemma~\ref{l:unif.cycles} and Lemma~\ref{l:unif.paths} can be combined to rule out cycles that are lighter than the targeted regime
for the mean weight.

\begin{ppn}\label{p:too.light}
For all \s{\ep>0} there exists \s{C=C(\ep)>0} such that
	\[
	\P\left(
		\begin{array}{c}
		\exists\text{ cycle }\CC
		: \mw(\CC) \le c_\circ-C/(\log n)^3\\
		\text{ and }
		\len(\CC)\ge (\log n)^2
		\end{array}
		\right) \le\ep.
	\]

\begin{proof}
From Lemma~\ref{l:many.subpaths}, assuming \s{A\geq 2},
a \s{c}-light \s{k}-cycle
either contributes to \s{Z^k_c(A)},
or has a subpath contributing to
\s{\tZ^\ell_c(A)} for some \s{\ell\le k}.
By Lemmas~\ref{l:unif.cycles}~and~\ref{l:unif.paths},
if we take \s{A=A_\circ} and \s{\gm} a large positive constant, then
	\[\P(Z^{\ge A^2}_c>0)
	\le
	\E Z^{\ge A^2}_c(A)
	+\E \tZ^{}_c(A)
	\lesssim \sum_{\ell\ge1}
		\f{e^{-\gm \ell/A^3}}{A^3}
	\asymp 1/\gm\,,
	\]
so by taking \s{\gm\asymp 1/\ep} we can ensure \s{\P(Z^{\ge A^2}_c>0)\le\ep}.
Since \s{c/c_\circ = \exp(-\gm/A^3)} and \s{c_\circ\asymp1}, the proposition follows with \s{C\asymp 1/\ep}.
\end{proof}
\end{ppn}

\subsection{Non-uniform cycles and long cycles}\label{sec:non-uniform-long}

Recall that in Lemma~\ref{p:too.unif} above we ruled out cycles that are too uniform
in the targeted regime for the mean weight.  We now prove the complementary assertion that cycles cannot be too non-uniform in this regime:

\begin{lemma}\label{c:too.non.unif}
For all \s{\ep,C>0}
there exists \s{\Delta=\Delta(\ep,C)>0} large enough that
	\[\P\left(
	\begin{array}{c}
	\exists\text{ cycle }\CC : \mw(\CC)\le
		c_\circ + C/(\log n)^3\\
	\text{ and \s{\CC}
		is not \s{(\log n+\Delta)}-uniform}
	\end{array}
	\right)
	\le\ep.\]

\begin{proof}
We can set \s{\gamma=1} and \s{A'\equiv\log n-\Delta'} with \s{\Delta'\asymp C} and have $c \equiv 1/(e \lambda_{A'}) \exp(-\gamma/(A')^3) \geq c_\circ + C/(\log n)^3$.  Then Lemma~\ref{l:unif.paths} gives
\[\E \tZ_c(A')\lesssim \frac{n}{(A')^3 e^{A'}} \frac{1}{1-e^{-\gamma/(A')^3}} \asymp n/e^{A'} = e^{\Delta'}\,.\]
Let \s{A=\log n+\Delta \ge A'}.  By Definition~\ref{d:unif} and Lemma~\ref{l:many.subpaths} together with Markov's inequality,
	\[
	\P(Z_c > Z_c(A))
	\le \P \big(\tZ_c(A') \ge
		1+\lfloor A-A' \rfloor \big)
	\le \f{\E \tZ_c(A')}
		{1+\lfloor \Delta+\Delta' \rfloor}
	\lesssim \f{e^{\Delta'}}{(\Delta+\Delta')}.
	\]
By taking \s{\Delta \asymp e^{\Delta'}/\ep} we can ensure that this probability is sufficiently small.
\end{proof}
\end{lemma}

In the remainder of this section we argue that there are no cycles in the targeted weight regime that have length much longer than \s{(\log n)^3}. In view of Lemmas~\ref{p:too.unif}~and~\ref{c:too.non.unif}, it remains to consider cycles which are neither too uniform nor too non-uniform:

\begin{lemma}\label{l:light.and.too.long}
For all \s{\ep,C,\Delta>0} there exists \s{B=B(\ep,C,\Delta)>0} large enough that
	\[
	\P\left(
	\begin{array}{c}
	\exists\text{ cycle }\CC:
	c_\circ - C/(\log n)^3 \le
	\mw(\CC) \le c_\circ + C/(\log n)^3,\\
	\text{\s{\CC} is \s{(\log n+\Delta)}-uniform
		but not \s{(\log n-\Delta)}-uniform},\\
	\text{and }\len(\CC)\ge B(\log n)^3
	\end{array}
	\right)
	\le\ep.
	\]

\begin{proof}
We denote
	\[c^\pm \equiv  c_\circ\pm C/(\log n)^3,\quad\quad
	A^\pm \equiv \log n \pm \Delta\,.\]
Consider a \s{k}-cycle with weights \s{w_1,\ldots,w_k} summing to \s{k\mw/n}, where \s{c^-\le\mw\le c^+},
which is \s{A^+}-uniform but not \s{A^-}-uniform.
Let \s{X} denote the corresponding untilted bridge, with increments \s{(nw_i/\mw-1)}.
For convenience we shift \s{X} by a constant so that its minimum is zero.
Then its maximum is in \s{[A^-,A^+]}.

Let \s{H} denote the ceiling of the maximum of \s{X}. Write \s{H\equiv \lceil A_\circ\rceil+\Delta_H} (so \s{\Delta_H} is integer-valued). Let \s{A' \equiv \lceil A_\circ\rceil-\Delta' } be the largest integer for which \s{A' \le A^- -2} and
	\beq\label{e:condition.on.delta.prime}
	c^+ e \lm_A
	\le \exp\{ -1/A^3 \}
	\quad\quad\text{for all $A$ such that }
	(\log n)/2
	\le A \le A'+2.
	\eeq
The constraint \s{A'\leq A^- -2} ensures \s{\Delta'+\Delta_H\geq 1}. It follows from \eqref{e:lambda.A.asymptotic} that \s{\Delta' \asymp\max(C,\Delta)}. Define 
	\[
	\begin{array}{rl}
	x \hspace{-6pt}&\equiv H-(A'-2)
	=\Delta_H+\Delta'+2,\\
	y \hspace{-6pt}&\equiv A'-2
	= H-(\Delta_H+\Delta'+2).
	\end{array}
	\]
For fixed \s{C} and \s{\Delta} when \s{n} is big enough we have \s{0<x<y<H}.

For a sequence of times \s{0<t_1<\ldots<t_q\le k}
with \s{q} even, let \s{E[t_{1:q}]} denote the event that the \s{t_i} partition the bridge \s{X} into alternating up-crossings and down-crossings: that is,
	\[
	E[t_{1:q}]
	=
	\set{U[t_{i-1},t_i]\text{ for all \s{i} even}}
	\cap
	\set{D[t_{i-1},t_i]\text{ for all \s{i} odd}}
	\]
where \s{U[a,b]} and \s{D[a,b]} indicate the events of up-crossing and down-crossing on \s{[a,b]}:
	\[\begin{array}{rl}
	U[a,b]
	\hspace{-6pt}&\equiv
	\set{ X_a<x\text{ and }b=\min\set{t>a : X_t>y} }\\
	D[a,b]
	\hspace{-6pt}&\equiv\set{ X_a>y\text{ and }b=\min\set{t>a: X_t<x} }.
	\end{array}
	\]
(with cyclic time indexing).  The proof idea is as follows: any cycle with a very large number of crossings will give an excessive contribution to \s{\tZ_c(A'')} where \s{A'' \equiv y-x}, so we can assume a bounded number of crossings. We then argue that a long cycle with a bounded number of crossings is unlikely to occur because its range is effectively reduced from \s{A} to \s{A'}.

Each down-crossing corresponds to a subpath with \s{c}-excedance \s{<-A''}. By Lemma~\ref{l:many.subpaths},
from each such down-crossing we can extract a
contribution to \s{\tZ_{c^+}(A'')}. From Lemma~\ref{l:unif.paths} and \eqref{e:condition.on.delta.prime}, the expected number of cycles with at least \s{Q/2} down-crossings is
	\[
	\leq \f{\E \tZ_{c^+}(A'')}{Q/2}
	\lesssim \f{n}{Q e^{A''}}
		\sum_{\ell\ge1} \f{(c^+e\lm_{A''}
			)^\ell}{(A'')^3}
	\lesssim \f{n}{Q e^{A''}}
	\lesssim \f{e^{\Delta+2\Delta'}}{Q},
	\]
which can be made arbitrarily small by taking \s{Q} large.

It therefore remains to control the cycles with $q<Q$.
To this end, we shall first bound the probability for an exp-minus-one walk \s{Y} started at \s{Y_0=X_0} --- with no conditioning on the value of \s{Y_k} --- to belong to the event \s{E[t_{1:q}]}.
\begin{enumerate}[1.]
\item First consider the probability of a length-\s{t} down-crossing, \s{D[a,a+t]}, conditioned on having just completed an up-crossing at time \s{a}. This means that \s{\lceil Y_a \rceil = y+h} for some integer \s{1\le h\le H-y}. Conditional on \s{h}, the probability to make a down-crossing is upper bounded by the probability that an exp-minus-one walk started exactly at \s{y+h} will travel to \s{[x-1,x+1]} in \s{t} steps without exiting the interval \s{[x-1,H+1]}. It follows that the probability of down-crossing \s{D[a,a+t]}, conditioned on having just completed an up-crossing at time \s{a}, is upper bounded by \s{p_\text{dn}(t)} where
	\[
	p_\text{dn}(t)
	= \max_{1\le h\le H-y}
	\int_0^2
	p^t_{A'}(y+h-(x-1),u) \,du.
	\]
Recalling \s{H-y = \Delta_H+\Delta'+2} and
applying \eqref{e:all.times.large.jump} gives
	\[p_\text{dn}(t)
	\lesssim \f{(\lm_{A'})^t(\Delta_H+\Delta')}{(A')^3}.
	\]
\item Similarly, consider the probability of a length-\s{t} up-crossing \s{U[a,a+t]}, conditioned on having just completed a down-crossing at time \s{a}. This is upper bounded by the probability \s{p_\text{up}(t)} that an exp-minus-one walk started exactly at the integer \s{x} will travel to \s{[y,H+1]} in \s{t} steps, and remain in the interval \s{[0,y+1]\subset[0,A']} up to time \s{t-1}.
Conditioning on the position \s{y+1-u} of the walk at time \s{t-1} and the position \s{z} at time \s{t} gives
	\[p_\text{up}(t)
	\le \int_y^{H+1}
		\int_0^{y+1}
			\f{p^{t-1}_{A'}(x,y+1-u)}
				{ e^{(z-y)+u} }
			\,du
			\, dz
	\lesssim 
		\int_0^{y+1}
			\f{p^{t-1}_{A'}(x,y+1-u)}{e^u}
			\,du.\]
If \s{t\gtrsim (\log n)^2}, then \eqref{e:all.times.kernel.ubd} gives
	\[
	p_\text{up}(t)
	\lesssim
	\f{(\lm_{A'})^t x}{(A')^3}
	\int_0^{y+1} \f{u+1}{e^u} \,du
	\lesssim
	\f{(\lm_{A'})^t (\Delta_H+\Delta')}{(A')^3}
	\]
If \s{t\lesssim (\log n)^2},
the contribution from \s{u\le y/2} satisfies the same bound by applying \eqref{e:all.times.large.jump} in place of \eqref{e:all.times.kernel.ubd},
while the contribution from \s{u\ge y/2} is negligible in comparison.
\end{enumerate}
\bigskip
Conditioned on all these crossings, the probability that \s{|Y_k-Y_0|\le\ep} is \s{\lesssim\ep},
whereas unconditionally
we have \s{|Y_k|\le\ep} with probability \s{\asymp \ep/k^{1/2}}.
Altogether we find
	\[\P(X\in E[t_{1:q}])
	= \lim_{\ep\downarrow0}
		\f{\P(Y\in E[t_{1:q}] ; |Y_k|\le\ep)}
			{\P(|Y_k|\le\ep)}
	\lesssim k^{1/2} \f{(\lm_{A'})^k (\Delta_H+\Delta')^{q} e^{O(q)}}{(A')^{3q}}.\]
Therefore, the probability to have
 a cycle \s{\CC} with \s{|\mw(\CC)-c_\circ|\le C/(\log n)^3}, which is \s{A^+}-uniformly light
but not \s{A^-}-uniformly light,
has length \s{\len(\CC)\ge L=B(\log n)^3}, and \s{q\le Q} is, using Markov's inequality
and Lemma~\ref{l:basic.first.moment},
	\[\lesssim
	\sum_{k\ge L} \E Z^k_{c^+}
	\sum_{q\le Q}
	\sum_{0<t_1<\ldots<t_q\le k}
	\P(X\in E[t_{1:q}])
	\lesssim
		 \sum_{k\ge L}
		\f{(c^+ e \lm_{A'})^k}{k^{3/2}}
	\sum_{q\le Q}
	\f{k^{1/2} [(\Delta_H+\Delta')k]^q  e^{O(q)}}{(A')^{3q}}\,.\]
We can assume \s{B\gtrsim 1/(\Delta_H+\Delta')} to make the summand increasing in $q$.
Recalling \eqref{e:condition.on.delta.prime}, this in turn is
	\[\lesssim
		\f{(A')^3  (\Delta_H+\Delta')^Q  e^{O(Q)} }{B(\log n)^3}
	\sum_{k\ge L}
		\f{Q [k/(A')^3]^Q}{\exp\{ k/(A')^3 \}}
		\f{1}{(A')^3}
	\lesssim
		\f{Q  (\Delta_H+\Delta')^Q e^{O(Q)}}{B}
		\int_{B}^\infty
		\f{z^Q}{e^z} \,dz.\]
This can be made arbitrarily small by taking \s{B} large, and the result follows.
\end{proof}
\end{lemma}

\begin{ppn}\label{p:light.and.too.long}
For all \s{\ep,C>0} there exists \s{B=B(\ep,C)>0} large enough that
	\[
	\P\left(
	\begin{array}{c}
	\exists\text{ cycle }\CC:
	c_\circ - C/(\log n)^3 \le
	\mw(\CC) \le c_\circ + C/(\log n)^3\\
	\text{and }\len(\CC)\ge B(\log n)^3
	\end{array}
	\right)
	\le\ep.
	\]

\begin{proof}
Follows by combining
Lemmas~\ref{p:too.unif},~\ref{c:too.non.unif},~and~\ref{l:light.and.too.long}.
\end{proof}
\end{ppn}

\begin{proof}[Proof of Theorem~\ref{t:first}]
Proposition~\ref{p:too.light} rules out light supercritical cycles of length \s{\geq (\log n)^2}, and
Lemma~\ref{l:log.squared} rules out the ones shorter than \s{(\log n)^2}, which implies \eqref{e:t.first.too.light}.
Proposition~\ref{p:light.and.too.long} rules out cycles of length \s{\gg(\log n)^3} in the regime \s{|\mw-c_\circ|\lesssim 1/(\log n)^3}, which together with \eqref{e:t.first.too.light} implies \eqref{e:t.first.light.and.too.long}.
\end{proof}

\section{Variance of typical-profile cycles}
\label{s:second.moment}

In this section we complete the proof of our main result Theorem~\ref{t:main}, the key ingredient of which is to demonstrate that cycles of length \s{\asymp (\log n)^3} exist in the regime \s{|c - c_\circ| \lesssim 1/(\log n)^3}.  (We also argue that it is unlikely that much shorter cycles exist in this weight range.)  We prove the existence of these cycles by a second-moment computation on a set of ``good'' cycles, as we now formally define.

\subsection{Good cycles} Recall \s{A_\circ\equiv\log n}, and let \s{A\equiv \lceil A_\circ\rceil-\Delta} for \s{\Delta} a large positive integer, with \s{\Delta\le (\log n)/2}. Choose \s{c\equiv c_A} such that \s{ce\lm_A=1}.

\begin{dfn}\label{d:good}
A \s{k}-cycle \s{\CC}, with weights \s{w_1,\ldots,w_k} summing to \s{k\mw/n}, is termed \s{\Delta}-\emph{good} if
\begin{enumerate}[(i)]
\item \s{c(1-1/k) \le\mw\le c}
	for \s{c=c_A}, \s{A= \lceil A_\circ\rceil-\Delta};
\item the process \s{X} with increments
\s{(nw_i/c-1)} has range \s{\le A-2},
\item the recentered process \s{\widetilde{X}}
(as in \eqref{e:X.recentered.around.half.A})
has \s{\le \Delta \times \delta_A(x)^4}
visits to \s{(x-1,x]},
for each \s{1\le x\le A}.
\end{enumerate}
\end{dfn}
Let \s{\bm{\Omega}_\Delta(k)} be the number of \s{\Delta}-good \s{k}-cycles \s{\CC}, and let
\[
{\bm\Omega}_\Delta \equiv \sum_{k=A^3/\Delta}^{A^3\Delta} {\bm\Omega}_\Delta(k)\,.
\]
We show by the moment method that for large~\s{\Delta}, \s{\bm{\Omega}_\Delta} is positive with large probability.
To this end we first argue that \s{\bm{\Omega}_\Delta} is large in expectation:

\begin{lemma}
\label{l:large.moment}
For \s{\Delta} large,
\s{\E\bm{\Omega}_\Delta \asymp\Delta}.

\begin{proof}
For \s{A^2\lesssim k \lesssim n^{1/2}},
consider a \s{k}-cycle \s{\CC}
 that satisfies properties (i) and (ii). The conditional probability that (iii) \emph{fails} is, by taking a union bound over \s{1\le x\le A} and applying Lemma~\ref{l:expected.profile} with Markov's inequality,
	\[\lesssim \sum_{x=1}^A
	\f{ 1}
		{\Delta\de_A(x)^2}
	\lesssim 1/\Delta.\]
It follows,
assuming \s{\Delta} large enough, that the expectation of \s{\bm{\Omega}_\Delta(k)}
is equal up to constants to the expected number of \s{k}-cycles satisfying properties (i) and (ii) only. By Lemma~\ref{l:unif.cycles},
	\[
	\E \bm{\Omega}_\Delta(k)
	\asymp
	\E[Z^k_c(A)-Z^k_{c(1-1/k)}(A)]
	\asymp
	\f{(ce\lm_A)^k}{A^3} =\f{1}{A^3}.
	\]
Summing over \s{\Delta^{-1}\le k/A^3 \le\Delta} proves the claim.
\end{proof}
\end{lemma}

\subsection{Variance bound}

The crux of the proof is in the following lemma which argues that the count of good cycles has low variance.

\begin{lemma}\label{l:conditional}
For any cycle \s{\CC},
	\[
	\sum_{\CC'\ne\CC}
	\P(\text{\s{\CC'} is \s{\Delta}-good} \,|\,
		\text{\s{\CC} is \s{\Delta}-good})
	\le \E\bm{\Omega}_\Delta+ o_\Delta(1)
	\]
where \s{o_\Delta(1)}
indicates an error tending to zero in the limit
of \s{\Delta\to\infty}.

\begin{proof}
Fix \s{\Delta} large, and fix a \s{k}-cycle \s{\CC}.
 Condition on the event that \s{\CC} is good
(from here on, ``good'' means \s{\Delta}-good), and write
	\[
	\bm{\Sigma}(\CC)
	\equiv
	\sum_{\CC'\ne\CC}
	\P(\text{\s{\CC'} is good} \,|\,
		\text{\s{\CC} is good}).
	\]
If a cycle \s{\CC'} is disjoint from \s{\CC} then the events \s{\set{\text{\s{\CC} is good}}} and \s{\set{\text{\s{\CC'} is good}}} are independent: thus, the contribution
to \s{\bm{\Sigma}(\CC)}
 from all \s{\CC'\cap\CC=\emptyset} is at most \s{\E\bm{\Omega}_\Delta}. It remains then to argue that the contribution from cycles intersecting with \s{\CC} is \s{o_\Delta(1)}.

To this end, consider \s{\CC'\ne\CC}
such that \s{\CC \cap \CC'} consists of \s{q} shared segments, with \s{q\ge1}. Let \s{r} count the vertices in \s{\CC'\setminus\CC}. The cycle \s{\CC'} is partitioned into alternating \emph{in-segments} (contained in \s{\CC}) and \emph{out-segments} (edge-disjoint from \s{\CC}); the combined length of all the out-segments is \s{r+q}.
Label the in-segments \s{[u_j,v_j]}, \s{1\le j\le  q}, in order of their traversal by \s{\CC'} --- \s{\CC} may traverse the segments in a different order.
Let \s{\mathbf{n}_\text{out}(r,q)}
count the number of
such cycles \s{\CC'} for a given tuple \s{(\CC,r,q,(u_j,v_j)_{1\le j\le q})},
	\[
	\mathbf{n}_\text{out}(r,q)
	\le n^r \binom{r+q-1}{q-1}
	\le
	n^r (r+q)^{q-1}\]
(since \s{\CC'} is given by choosing
an ordered sequence of \s{r} vertices,
then dividing them in to \s{q} groups).

Let \s{X'} be the process with increments \s{(nw'_i/c-1)}, where \s{w'_i} are the edge weights on \s{\CC'}. The edge weights in \s{\CC'\setminus\CC} are still i.i.d.\ unit-rate exponential random variables after the conditioning on \s{\CC}. For \s{1\le j\le q} let integers \s{y_j,z_j} be defined by
	\[
	\begin{array}{rll}
	y_j\hspace{-6pt}&=\lceil X'_{v_j} \rceil
	-\lceil X'_{u_j}\rceil
	&\text{(discretized increment over \s{j}-th in-segment);}\\
	z_j\hspace{-6pt}&=\lceil X'_{u_{j+1}}\rceil
	-\lceil X'_{v_j}\rceil
	&\text{(discretized increment over \s{j}-th out-segment).}
	\end{array}
	\]
Write \s{\bm{y}\equiv(y_1,\ldots,y_q)}
and likewise \s{\bm{z}\equiv(z_1,\ldots,z_q)}.
Let \s{\overline{y}} denote the average of the \s{y_j},
and likewise \s{\overline{z}}
the average of the \s{z_j}. For \s{\CC'} to satisfy property (i), we must have \s{q(\overline{y}+\overline{z})\in\set{-1,0}}.
For each \s{j}, 
\s{y_j+A\ge1} and
\s{-z_j+A\ge1}. Thus,
given \s{\overline{y}},
the number of choices for compatible \s{(\bm{y},\bm{z})} is at most
	\[\mathbf{a}_q(\overline{y})
	\equiv \binom{q(A+\overline{y})-1}{q-1}
	\sum_{i=0}^1
		\binom{q(A+\overline{y})+i-1}{q-1}
	\le [e^{O(1)}(A+\overline{y})]^{2q}.\]
Next, by property (iii), the number of choices of \s{(u_j,v_j)} consistent with \s{\bm{y}} is at most
	\[\prod_{j=1}^q\Big\{
	 \Delta^2
	\sum_x \delta_A(x)^4 \delta_A(x+|y_j|)^4\Big\}
	\le
	\Delta^{2q}
	\prod_{j=1}^q\Big\{
	(A-|y_j|) 
	\max_x
	\big[
	 (x+1)^4 (A-|y_j|-x+1)^4
	 \big]\Big\},\]
where we have used that there are \s{\le A-|y_j|}
choices of \s{x} for which the summand is positive.
For each \s{j} the maximum is achieved with \s{x+1=(A-|y_j|)/2}, and combining with Jensen's inequality the above is
	\[\le
	\Delta^{2q}\prod_{j=1}^q (A-|y_j|)^9
	\le\Delta^{2q} \Big(
		A- q^{-1}\sum_j |y_j|
		\Big)^{9q}
	\le \Delta^{2q}(A+\overline{y})^{9q}
	\equiv \mathbf{b}_q(\overline{y}).\]
Lastly, let \s{\mathbf{p}(\bm{z})} be the probability that \s{X'} has increments
as specified by \s{\bm{z}},
with range at most \s{A} (otherwise \s{\CC'} would violate property (ii)). Writing \s{\ell_j} for the length of the \s{j}-th out-segment, a similar calculation as in the proof of Lemma~\ref{l:unif.paths} gives
	\[
	\mathbf{p}(\bm{z})
	\le 
	\prod_{j=1}^q\bigg\{
	O(c/n)
	\f{f_{\ell_j}[ (\ell_j+z_j)c/n]}
	{f_{\ell_j}(\ell+z_j)}
	\max_x p_A^{\ell_j}(x,x+z_j)
	\bigg\}.\]
Applying \eqref{e:all.times.large.jump}
and making some straightforward manipulations
(recalling \s{ce\lm_A=1}), we arrive at the conclusion that for any \s{\bm{z}} consistent with \s{\bm{y}},
	\[\mathbf{p}(\bm{z})\le
	\Big(\f{ce\lm_A}{n}\Big)^{r}
	\Big[
	\f{e^{O(1)}
	ce\lm_A}{ n e^{-\overline{z}} }
	\f{(A-\overline{z})^2}{A^3}
	\Big]^q \le e^{O(q)}
	\mathbf{c}_{r,q}(\overline{y}),
	\quad
	\mathbf{c}_{r,q}(\overline{y})\equiv
	\f{1}{n^r}
	\Big[
	\f{(A+\overline{y})^2}
	{ n e^{\overline{y}} A^3}
	\Big]^q.\]
Combining these estimates gives
	\[
	\bm{\Sigma}(\CC)-\E\bm{\Omega}_\Delta
	\le e^{O(q)} \sum_{q,r,\overline{y}}
	\mathbf{n}_\textup{out}(r,q)
	\mathbf{a}_q(\overline{y})
	\mathbf{b}_q(\overline{y})
	\mathbf{c}_{r,q}(\overline{y})
	\le \sum_{q,r,\overline{y}}
	\Big[
	\f{e^{O(1)} \Delta^2
	(A+\overline{y})^{13}
	}{ ne^{\overline{y}}  }
	\Big]^{q}
	\f{(r+q)^{q-1}}{A^{3q}}
	\]
where the sum is taken over \s{q\ge1}, \s{q\le r+q \le \Delta A^3}, and \s{|\overline{y}|\le A-2}
with \s{q\overline{y}} integer-valued. Making the change of variables \s{Y \equiv A+\overline{y}\ge2}, we find
	\[
	\bm{\Sigma}(\CC)-\E\bm{\Omega}_\Delta
	\le
	\sum_{q, Y}
	\Big[
	\f{ \Delta^3
	}{ n/e^A }
	\f{Y^{13}}{e^Y}
	\Big]^q
	\le
	\sum_{q, Y}
	\Big[
	\f{ \Delta^3
	}{ e^\Delta }
	\f{Y^{13}}{e^Y}
	\Big]^q
	\lesssim \f{ \Delta^3
	}{ e^\Delta },
	\]
which tends to zero in the limit \s{\Delta\to\infty} as claimed.
\end{proof}
\end{lemma}

\begin{cor}\label{c:whp}
In the limit \s{\Delta\to\infty},
\s{\P(\bm{\Omega}_\Delta=0)\to0}.
\begin{proof}
By Chebychev's inequality and Lemma~\ref{l:conditional},
	\[
	\P(\bm{\Omega}_\Delta=0)
	\le\f{\text{Var}\,|\bm{\Omega}_\Delta|}{(\E|\bm{\Omega}_\Delta|)^2}
	\le \f{\E|\bm{\Omega}_\Delta| + o_\Delta(1)}{(\E|\bm{\Omega}_\Delta|)^2}.
	\]
By Lem.~\ref{l:large.moment} the right-hand side tends to zero in the limit \s{\Delta\to\infty}.
\end{proof}
\end{cor}

\subsection{Conclusion}

We conclude this section with the proof of our main result.

\begin{proof}[Proof of Theorem~\ref{t:main}]
Theorem~\ref{t:first}
and Corollary~\ref{c:whp} together imply that
we can choose \s{C_1=C_1(\ep)} sufficiently large so that
	\[
	\P\left(\left.
	\begin{array}{c}
	c_\circ-C_1/(\log n)^3
	\le n\W_n
	\le c_\circ+C_1/(\log n)^3\\
	\text{and
	\s{(\log n)^2/C_1\le \LL_n
	\le C_1(\log n)^3}}
	\end{array}
	\,\right|\,n\W_n>1/e
	\right)\ge1-\ep.
	\]
Combining this with
Lemma~\ref{l:unif.cycles}
and Lemma~\ref{c:too.non.unif},
we can choose \s{C_2=C_2(\ep,C_1)} sufficiently large so that for any interval \s{I\subseteq[0,C_1]} with \s{|I| \le 1/C_2},
	\[
	\P\left(\begin{array}{c}
	\exists\text{ cycle }\CC:
	[\mw-c_\circ](\log n)^3
	\in I\\
	\text{or
	\s{\len(\CC)/(\log n)^3
	\in I}}
	\end{array}
	\right)\le\ep.
	\]
Combining these proves \eqref{e:main.len.wt} and \eqref{e:main.nondegenerate}.
\end{proof}

\section{Eigenfunctions of the one-minus-exp walk}\label{s:eig}

Recall that the one-minus-exp random walk is the real-valued stochastic process \s{(X_t)_{t\ge0}} whose jumps \s{X_{t+1}-X_t} are independent and identically distributed as \s{1-\Exp(1)}. In this section we study eigenfunctions of this walk with killing outside an interval \s{[0,H]}.

\subsection{Eigenfunctions and eigenvalues}\label{ss:eig.intro}
Consider the one-minus-exp walk killed outside \s{[0,H]}, and suppose that \s{g} is a left eigenfunction of this process with eigenvalue \s{\lm}: this means (see Remark~\ref{r:left.right}) that \s{g} is a function supported on \s{[0,H]}, not identically zero,
satisfying
	\beq\label{e:left.eigenvalue.eqn}
	\lm g(x) = e^{x-1}\int_0^H
		\Ind{u\ge x-1} g(u) e^{-u} \,du,\quad
		0\le x \le H.
	\eeq
As \s{g} is not identically zero,
\s{\set{x : g(x)\ne0}} is a non-empty subset of \s{[0,H]},
and we denote its infimum by \s{x_{\min}}. We then have
	\[\lm g(x) e^{1-x} = \int_0^H g(u) e^{-u}\,du,
	\quad 0\le x \le \min\set{H,1+x_{\min}}.\]
--- including for some 
\s{x_{\min} \le x \le \min\set{H,1+x_{\min}}}
for which \s{g(x)\ne0}. If the integral on the right-hand side vanishes, then \s{\lm=0}; otherwise \s{\lm\ne0}, \s{x_{\min}=0}, and (by rescaling) \s{g(x) = e^x} for \s{0\le x\le \min\set{H,1}}.
In particular, for \s{H\le1} we conclude the process has a unique non-zero eigenvalue \s{\lm=H/e} with associated  left eigenvector \s{g(x) = e^x \Ind{0\le x\le H}}.

We see from \eqref{e:left.eigenvalue.eqn}
that \s{g} is smooth on \s{(0,H)\setminus\Z},
with continuous derivatives up to order \s{k}
on \s{(0,H)\setminus\Z_{\le k}}. We can therefore differentiate \eqref{e:left.eigenvalue.eqn} to find
	\beq\label{eq:g-ODE}
	\lm g'(x)
	= \lm g(x) - g(x-1),\quad 1<x<H.
	\eeq
Take for the moment \s{H=\infty}. We can solve for \s{g} on the intervals
\s{[k,k+1]} (\s{k\in\Z_{\ge1}}) one at a time, as follows: if \s{g} satisfies \eqref{eq:g-ODE}, then
\s{h} satisfies \s{h'(x) = g(x-1)/\lm}
	for all \s{x>1}, where
	\beq\label{e:from.g.to.h}
	h(x)
	\equiv -g(x) + \int_1^x g(u)\,du,\quad x\ge1.
	\eeq
Knowing \s{g(x)} for \s{1\le x\le \overline{x}}
determines \s{h(x)} for \s{1\le x\le \overline{x}};
the reverse also holds since
	\[
	0=-g(x) + \int_1^x g(u)\,du,
	\quad 0\le x\le\overline{x}
	\]
has only the trivial solution \s{g=0}.
Suppose inductively that \s{g} (hence \s{h}) has been determined on \s{[0,k]} for \s{k\in\Z_{\ge1}}: then
we can determine \s{h} on the next interval \s{[k,k+1]}
by evaluating
	\[
	h(x) = h(k)+ \lm^{-1} \int_{k-1}^{x-1} g(u)\,du,
	\quad
	k\le x \le  k+1,
	\]
which determines \s{g} on \s{[k,k+1]}. This proves that in the case \s{H=\infty}, there is a unique continuous function \s{g= g_\lm} supported on \s{[0,\infty)} that satisfies \s{g(x)=e^x} for \s{0\le x\le1}, and satisfies \eqref{eq:g-ODE} for \s{x>1}. It is straightforward to verify that
	\beq\label{eq:g}
	g_\lm(x) \equiv e^x \sum_{k=0}^{\lceil x\rceil-1}
		\f{(x-k)^k}{(-\lm e)^kk!},\quad x\ge0
	\eeq
is such a function, so it must be the unique one.

\begin{figure}[h!]
\begin{center}
\includegraphics[width=0.45\textwidth]{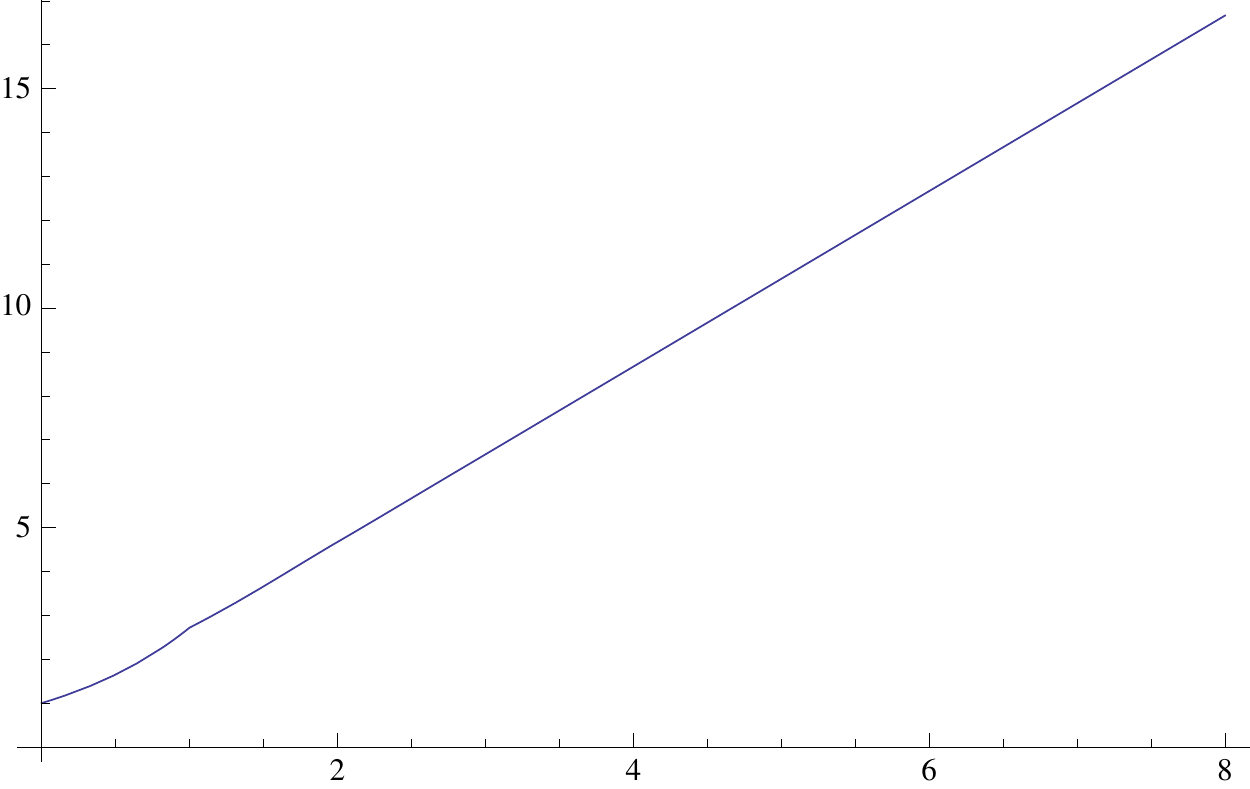}\hfill\includegraphics[width=0.45\textwidth]{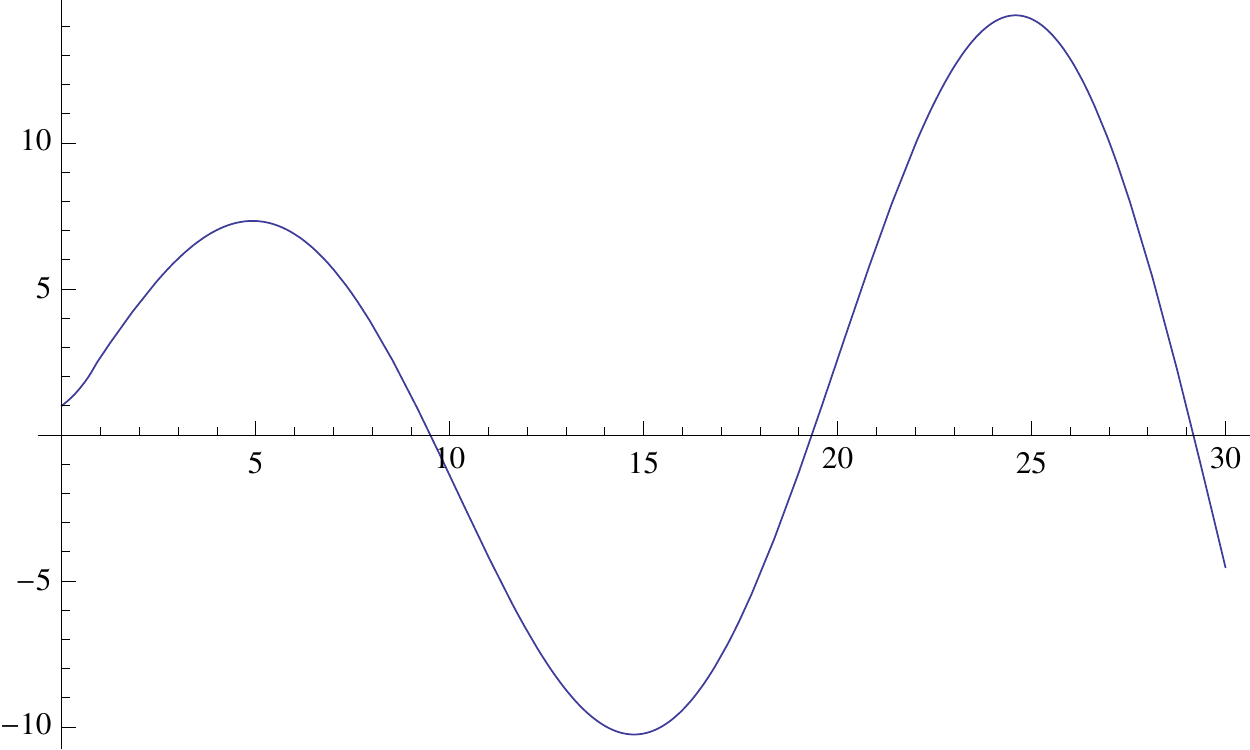}
\end{center}
\caption{The left plot shows $g_1(x)$, and the right plot shows $g_{0.95}(x)$.}
\label{fig:g_lambda(x)}
\end{figure}

If for finite \s{H} and \s{\lm\ne0}, \s{g} is a solution of \eqref{e:left.eigenvalue.eqn}, then \s{g} satisfies \eqref{eq:g-ODE}, and it follows from the preceding discussion that \s{g} must be the restriction of \s{g_\lm} to \s{[0,H]}. We calculate
	\[
	\f{1}{\lm e}
	\int_a^b g_\lm(u) e^{-u}\,du
	= \sum_{k=1}^{\lceil a\rceil}
		\f{(a+1-k)^{k}}{(-\lm e)^k k!}
	-\sum_{k=1}^{\lceil b\rceil}
		\f{(b+1-k)^{k}}{(-\lm e)^k k!}
	= \f{g_\lm(a+1)}{e^{a+1}}
		-\f{g_\lm(b+1)}{e^{b+1}},
	\]
so equation \eqref{e:left.eigenvalue.eqn} with \s{g=g_\lm}
and \s{1\le x\le H} simplifies to
	\[
	0
	= \f{g_\lm(x)}{e^x}
	-\f{1}{\lm e}\int_{x-1}^H g_\lm(u) e^{-u}\,du
	= \f{g_\lm(H+1)}{e^{H+1}},
	\]
therefore \s{g_\lm(H+1)=0}
which is a polynomial equation in \s{\lm} of degree \s{\lceil H\rceil}
(Fig.~\ref{fig:lambdas(H)}).

\begin{figure}[t!]
\begin{center}
\includegraphics[width=0.5\textwidth]{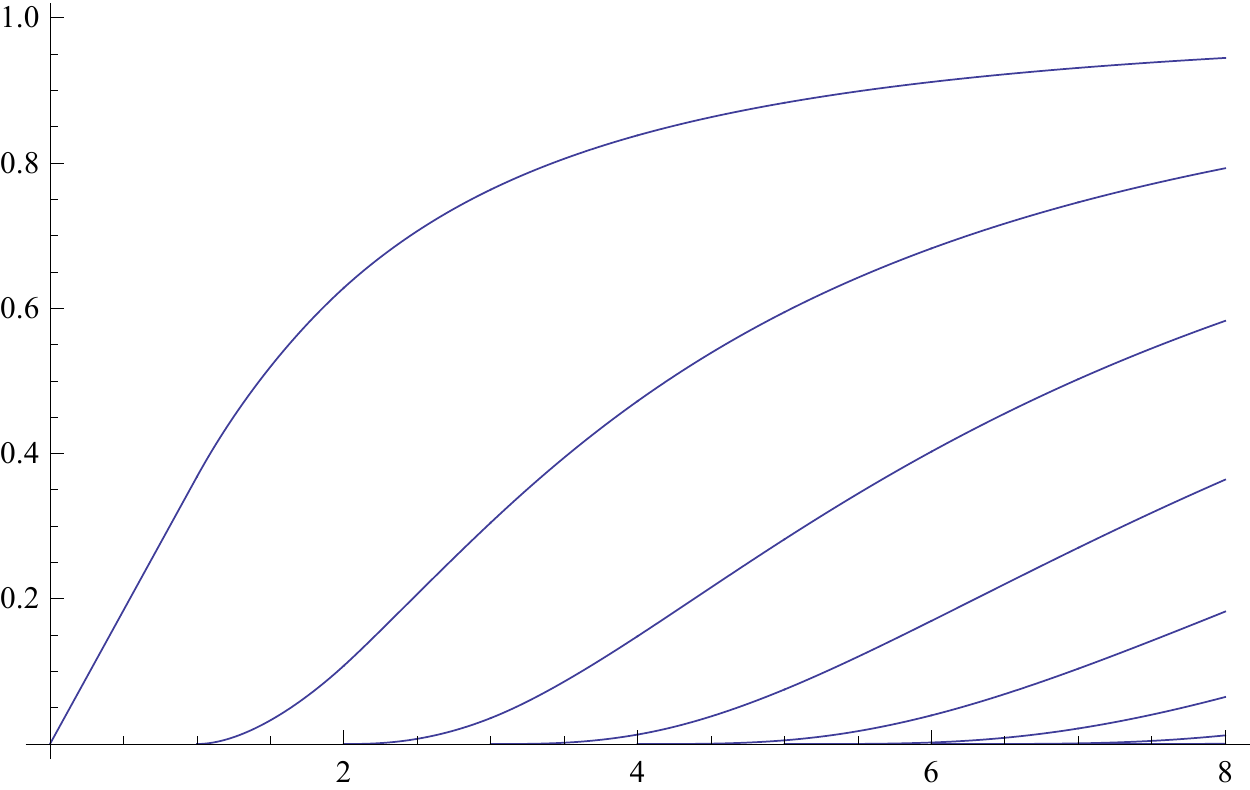}
\end{center}
\caption{The eigenvalues of the walk with two-sided killing as a function of \s{H}.}
\label{fig:lambdas(H)}
\end{figure}

\begin{rmk}\label{r:left.right}
We comment briefly on left and right eigenfunctions of the process. Formally,
the operator for the 
one-minus-exp walk \s{X_t} killed outside \s{[0,H]}
is given by
	\[
	(Kf)(x)
	\equiv \E_x[ f(X_1) \Ind{0\le X_1\le H}]
	= \int_0^H f(y) p(x,y) \,dy,\]
where \s{p} is the transition kernel from \s{x} to \s{y}.
The left action of \s{K} is the right action of the adjoint operator, \s{(K^\star f)(y) = (fK)(y)}. In the current setting, \s{p(x,y) = k(y-x)} where \s{k} is the density function of the \s{1-\Exp(1)} random variable. We therefore have
	\[
	(K^\star f)(x)
	= \int_0^H f(y) k(x-y)\,dy
	= \int_0^H f(H-y) k(y-(H-x))\,dy
	= (R_H KR_H f)(x),
	\]
where  \s{R_H f(x) \equiv f(H-x)}. The reflection \s{R_H} is involutive, and it relates the left and right eigenfunctions of \s{K}: \s{g} is a left eigenfunction of \s{K} (i.e.\ a right eigenfunction of \s{K^\star}) if and only if \s{R_H g} is a right eigenfunction of \s{K} with the same eigenvalue. If \s{f_1,f_2} are right eigenfunctions of \s{K} with non-zero eigenvalues \s{\lm_1,\lm_2} then, writing \s{\langle\cdot,\cdot\rangle} for the \s{L^2[0,H]} inner product,
	\[
	\langle R_H f_1,f_2 \rangle
	= (\lm_2)^{-1} \langle R_H f_1, K f_2 \rangle
	= (\lm_2)^{-1} \langle K^\star R_H f_1, f_2 \rangle
	= (\lm_1/\lm_2)\langle R_H f_1,f_2 \rangle.
	\]
Consequently, if \s{\lm_1\ne\lm_2} then
\s{\langle R_H f_1,f_2 \rangle=0}.
\end{rmk}

\subsection{Series expansion}
We now review the solution obtained by Wright~\cite{wright} for general homogeneous
difference-differential equations
with constant coefficients,
	\beq\label{e:Lambda.equals.zero}
	\Lm(y)(x)
	\equiv\sum_{\mu=0}^m\sum_{\nu=0}^n
	a_{\mu\nu} y^{(\nu)}(x+b_\mu)=0, \quad
	a_{\mu\nu}\in\C, \
	0=b_0<b_1<\ldots<b_m,
	\eeq
where \s{y^{(\nu)}} denotes the \s{\nu}-th derivative of \s{y}. Observe that
	\[
	\Lm(e^{sx})
	= e^{sx} \tau(s)  \text{ where }
	\tau(s)\equiv\sum_{\mu=0}^m\sum_{\nu=0}^n
		a_{\mu\nu} e^{b_\mu s} s^\nu,
	\]
so if \s{\tau(s)=0} then \s{y(x)=e^{s x}} solves \s{\Lm(y)=0}. More generally, we can apply the formula
\s{(fg)^{(\nu)}
= \sum_{r=0}^\nu\binom{\nu}{r} f^{(r)} g^{(\nu-r)}}
to calculate
	\[
	\Lm(x^\ell e^{sx})
	= e^{sx}
	\sum_{\mu=0}^m\sum_{\nu=0}^n
	a_{\mu\nu}
	\sum_{r\ge0}
		\f{(\nu)_r (\ell)_r }{r!} 
		e^{b_\mu s}
		s^{\nu-r}
		(x+b_\mu)^{\ell-r},
	\]
where for \s{r\in\Z_{\ge0}},
 \s{(j)_r} denotes the falling factorial
\s{j(j-1)\cdots(j-r+1)}, which is zero if \s{r-j\in\Z_{\ge1}}. Expanding in powers of \s{x} gives
	\[
	\Lm(x^\ell e^{sx})
	= e^{sx}
	\sum_{j=0}^\ell
	\binom{\ell}{j}
	x^j
	\underbracket{
	\sum_{\mu=0}^m\sum_{\nu=0}^n
	a_{\mu\nu}
	\sum_{r\ge0}
	\f{(\nu)_r(\ell-j)_r}{r!}
	e^{b_\mu s} s^{\nu-r} (b_\mu)^{\ell-j-r}
	}_{\tau^{(\ell-j)}(s)}.
	\]
If \s{s} is a root of \s{\tau} of order \s{\ell+1},
meaning \s{0=\tau(s)=\tau'(s)=\ldots=\tau^{(\ell)}(s)},
then \s{\Lm(x^j e^{sx})=0} for the integers \s{0\le j\le \ell}.

Assume that \s{m,n\ge1}, and that the \s{m\times n} coefficient matrix \s{a_{\mu\nu}} contains a non-zero entry in each of the first and last rows and columns, which eliminates pure difference equations and pure differential equations. Under these assumptions, it is shown \cite{wright} that the general solution to \eqref{e:Lambda.equals.zero} is given by a limit of linear combinations of the solutions \s{x^j e^{s x}} described above; further, it is explained how to compute the coefficients of this linear combination given initial data \s{y^{(\nu)}(0)} (\s{0\le \nu<n}) and \s{y^{(n)}(x)} (\s{0\le x\le b_m}). Let
	\begin{align*}
	\bm{H}_1(s)
	&\equiv \sum_{\mu=1}^m \sum_{\nu=0}^n
		a_{\mu\nu} e^{b_\mu s}
		\int_0^{b_\mu}
		y^{(\nu)}(u) e^{-su}\,du,\\
	\bm{H}(s)&\equiv \bm{H}_1(s)
	+ \sum_{\mu=0}^m\sum_{\nu=1}^n
		a_{\mu\nu} e^{b_\mu s}
		\sum_{\iota=1}^{\nu}
		s^{\nu-\iota}
		y^{(\iota-1)}(0),
	\end{align*}
and let \s{F_x(s) \equiv  e^{sx}\bm{H}(s)/\tau(s)},
regarded as a function of \s{s\in\C}.
It is then shown that
	\beq\label{e:wright.formula}
	y(x) = \sum_{s} \res_{s} F_x
	\eeq
where \s{\res_{s} F_x} denotes the complex residue of \s{F_x} at a pole \s{s\in\C}, and the sum is taken over all poles \s{s}. To ensure convergence, the poles are arranged in an appropriate order \cite{wright} which is somewhat delicate in general; in our particular setting we will find below that the summation is absolutely convergent.

To avoid the singularity of \s{g'} at zero, let \s{y(x) \equiv g_\lm(x+\ep)}, so \s{y} satisfies
	\[
	\lm^{-1}y(x) - y(x+1) + y'(x+1)=0,\quad
	-\ep<x< \infty. 
	\]
This clearly corresponds to \eqref{e:Lambda.equals.zero}
with \s{m=n=1},
	\beq\label{eq:tau}
	a_{\mu\nu} = \bordermatrix{
	 & \scriptstyle\nu{=}0 & \scriptstyle\nu{=}1 \cr
	\scriptstyle\mu{=}0 & 1/\lambda & 0 \cr
	\scriptstyle\mu{=}1 & -1 & 1
	},\quad
	b_\mu = \bordermatrix{
	 & \cr
	\scriptstyle\mu{=}0 & 0 \cr
	\scriptstyle\mu{=}1 & 1 
	},\quad
	\tau(s)=1/\lambda - e^s + s e^s.
	\eeq
Note that \s{\tau'(s)=e^s s} and \s{\tau''(s)=e^s(s+1)}, so if \s{\tau(s)=0} then \s{s} is a simple root unless \s{s=0}, in which case \s{\lm=1} and \s{s=0} is a double root. From \eqref{eq:g} we have initial data
	\[
	y(x) = g_\lm(x+\ep)
	= \left\{\hspace{-4pt}\begin{array}{ll}
	e^{x+\ep}
		&\text{for } x+\ep\in[0,1],\\
	e^{x+\ep}[1-(x+\ep-1)/(\lm e)]
		&\text{for } x+\ep\in[1,2],
	\end{array}\right.
	\]
which we use to evaluate
	\[
	\bm{H}_1(s)
	= e^s\int_0^1 [-y(u)+y'(u)] e^{-su}\,du
	= \f{-(e^\ep-e^{\ep s})}{\lm(1-s)}, \quad
	\bm{H}(s)
	= \bm{H}_1(s) + e^{s+\ep},
	\]
where \s{\bm{H}_1(s)|_{s=1}} is understood to be \s{-\ep e^\ep}. If \s{\tau(s)=0} then \s{\bm{H}(s) = e^{(1+\ep)s}\ne0}, so the poles of \s{F_x(s) \equiv e^{sx}\bm{H}(s)/\tau(s)} correspond precisely to the zeroes of \s{\tau(s)}. We compute
	\[
	\res_s F_x
	= \left\{\begin{array}{ll}
	\displaystyle \lim_{z\to s} (z-s) F_x(z)
	= \f{e^{sx} \bm{H}(s)}{\tau'(s)}
	= \f{e^{(x+\ep)s}}{s}
		&\text{if \s{s} is simple pole,}\\
	\displaystyle
	\lim_{z\to 0} \f{d}{dz} [ z^2 F_x(z) ]
	= 2(x+\ep+1/3)
		&\text{if
			\s{\lm=1,s=0}}
	\end{array}\right.
	\]
(recall from above that these are the only two cases
for roots of \s{\tau}). Applying \eqref{e:wright.formula} gives
	\beq\label{e:g.lambda.sum.over.poles}
	g_\lm(x)
	=\sum_{\substack{s\in\C:\\
		\tau(s)=0}}
	\f{e^{sx}}{s}
	\text{ for }\lm\in(0,1), \quad
	g_1(x)
	= 2(x+1/3) 
	+ \sum_{\substack{s\in\C\setminus\set{0}:\\
		\tau(s)=0}}
	\f{e^{xs}}{s},\eeq
modulo issues of convergence to be addressed 
in the next subsection.

\subsection{Roots of \texorpdfstring{$\tau$}{\texttau} and Lambert \texorpdfstring{$W$}{W} function}

The roots of \s{\tau} can be expressed in terms of the Lambert \s{W} function, which has been very well studied and which we now briefly review.\footnote{This function also arose in the analysis \cite{mathieu-wilson} of the subcritical regime of the minimum mean-weight cycle,
for seemingly different reasons. It figures prominently in the analysis of random graphs near the phase
transition, e.g.\ \cite{JKLP}; see \cite[\S2]{Lambert-W} for additional applications. Some of the discussion in this section is adapted from~\cite{mathieu-wilson}.} The number of rooted spanning trees of the complete \s{n}-vertex graph is \s{n^{n-1}} (Cayley's formula). The \emph{tree function} \s{T} is the associated exponential generating function,
	\[ T(z)
	\equiv \sum_{n=1}^\infty
		\f{n^{n-1}}{n!} z^n.
	\]
From Stirling's formula (\s{n!\sim \sqrt{2\pi n}(n/e)^n}), the sum converges for all $|z|\leq 1/e$. It satisfies the relation \s{T(z) = z\exp\{ T(z)\}} (see e.g.\ \cite[Proposition~5.3.1]{stanley}). The \emph{Lambert \s{W} function} is defined by the equation
	\beq\label{e:def.Lambert.W}
	z = W(z) \, e^{W(z)}.
	\eeq
This is a multivalued function, with branches \s{W_k} naturally indexed by the integers \s{k\in\Z}; see \cite[\S4]{Lambert-W} and Fig.~\ref{f:lambert.w}.

\begin{figure}[h]
\includegraphics[width=0.5\textwidth]{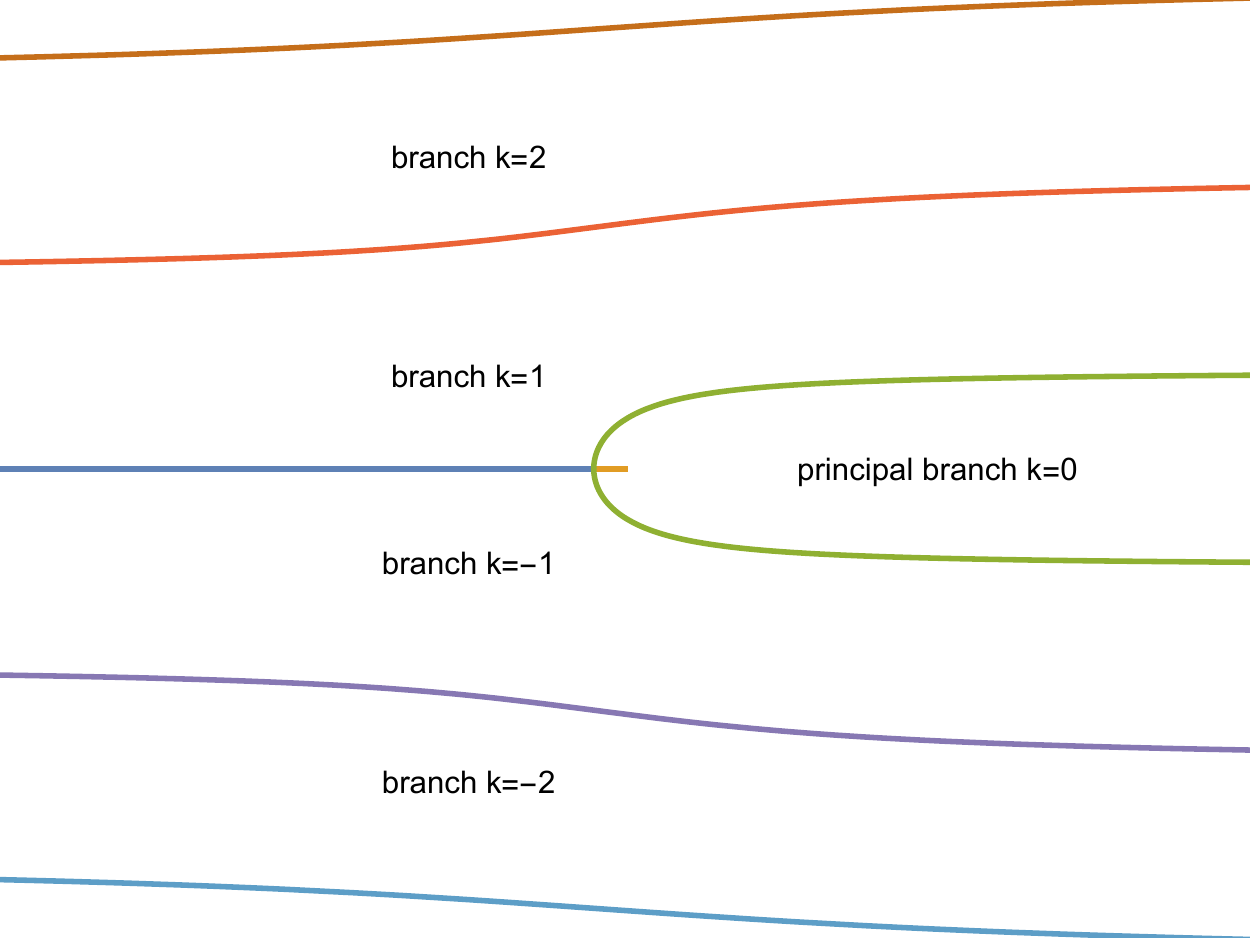}
\caption{(This is a duplicate of \cite[Fig.~4]{Lambert-W}.) The Lambert \s{W} function is defined as the (multi-valued) inverse of the function \s{f(w) = w e^w}. The curves in the figure show the preimage set \s{f^{-1}(\R_{<0})} --- they are a subset of \s{\R\cup\set{w:\xi=-\eta\cot\eta}} where \s{\xi\equiv\real w} and \s{\eta\equiv\imag w}. The curves naturally partition the \s{w}-plane into branches indexed by \s{k\in\Z} (the \s{k}-th branch is the image of \s{W_k}). Following convention \cite{Lambert-W}, branches are defined to be closed in the direction of increasing \s{\imag w}.}
\label{f:lambert.w}
\end{figure}

The principal branch \s{W_0} satisfies \s{W_0(z) = -T(-z)} for \s{|z|\leq 1/e}, and can be defined elsewhere by analytic continuation.
It is straightforward to check that \s{T(1/e) = 1}.
For \s{\delta\in\C} with \s{|e^{-\delta}|\le1}, we can use \s{T(1/e) = 1} and the relation \s{T(z) = z\exp\{ T(z)\}} to deduce that
	\beq\label{eq:treefunctionatcritical}
	T(1/e^{1+\delta})
	= 1 - \sqrt{2\delta} + O(\delta).
	\eeq
The branch \s{W_0} has a cut on \s{z\in(-\infty,-1/\ee)}; and in view of \eqref{eq:treefunctionatcritical}, it has a  square-root-type singularity near \s{z=-1/\ee}. 
Any other branch \s{W_k} (\s{k\ne0}) 
has a cut on \s{z\in(-\infty,0)} with a logarithmic singularity near \s{z=0}.

Recalling \eqref{eq:tau}, the solutions of \s{\tau(s)=0} are given by
	\beq\label{e:def.s.k.using.W.k}
	s_k=1+w_k,
	\quad w_k\equiv W_k(-1/(\ee\lambda)),
		\quad k\in\Z.\eeq
For \s{\lm\in(0,1]}, \s{s_k} is obtained by evaluating \s{W_k} precisely on the branch cut. Following the convention that branches are closed in the direction of increasing \s{\imag w} (\cite{Lambert-W}, see also Fig.~\ref{f:lambert.w}), we see that \s{s_k} and \s{s_{-k-1}} are complex conjugates for each \s{k\ge0}. When \s{\lm=1}, \s{s_0} and \s{s_{-1}} are obtained by evaluating \s{W_0} and \s{W_{-1}} at the branch point \s{-1/e}, giving the double root \s{s_0=s_{-1}=0}. There is a convergent series expansion \cite[eqn.~4.20]{Lambert-W} for each branch \s{W_k}; truncating the series and evaluating at \s{-1/(\ee\lambda)} gives, for positive integers \s{k},
	\beq\label{eq:s_k-asymptotic}
	\begin{array}{rl}
	s_k
	\hspace{-6pt}&=
	-\log[\pi\lm(4k+1)/2]
 	+i\pi(4k+1)/2
	+ O(1+k^{-1}{\log k}), \vspace{2pt}\\
	s_{-k-1}
	\hspace{-6pt}&=
	-\log[\pi\lm(4k+1)/2]
 	-i\pi(4k+1)/2
	+ O(1+k^{-1}{\log k}).
	\end{array}\eeq
Thus, for any \s{x\in\R} there exists a finite constant \s{C_x} such that \s{|\exp\{s_k x\}/s_k| \le C_x/|k|^{1+x}} for all \s{k\in\Z}, excluding the case \s{s_0=s_{-1}=0}. It follows that for \s{x>0}, the summations in \eqref{e:g.lambda.sum.over.poles} are absolutely convergent as claimed.

\begin{figure}[t!]
\begin{center}
\includegraphics[width=0.5\textwidth]{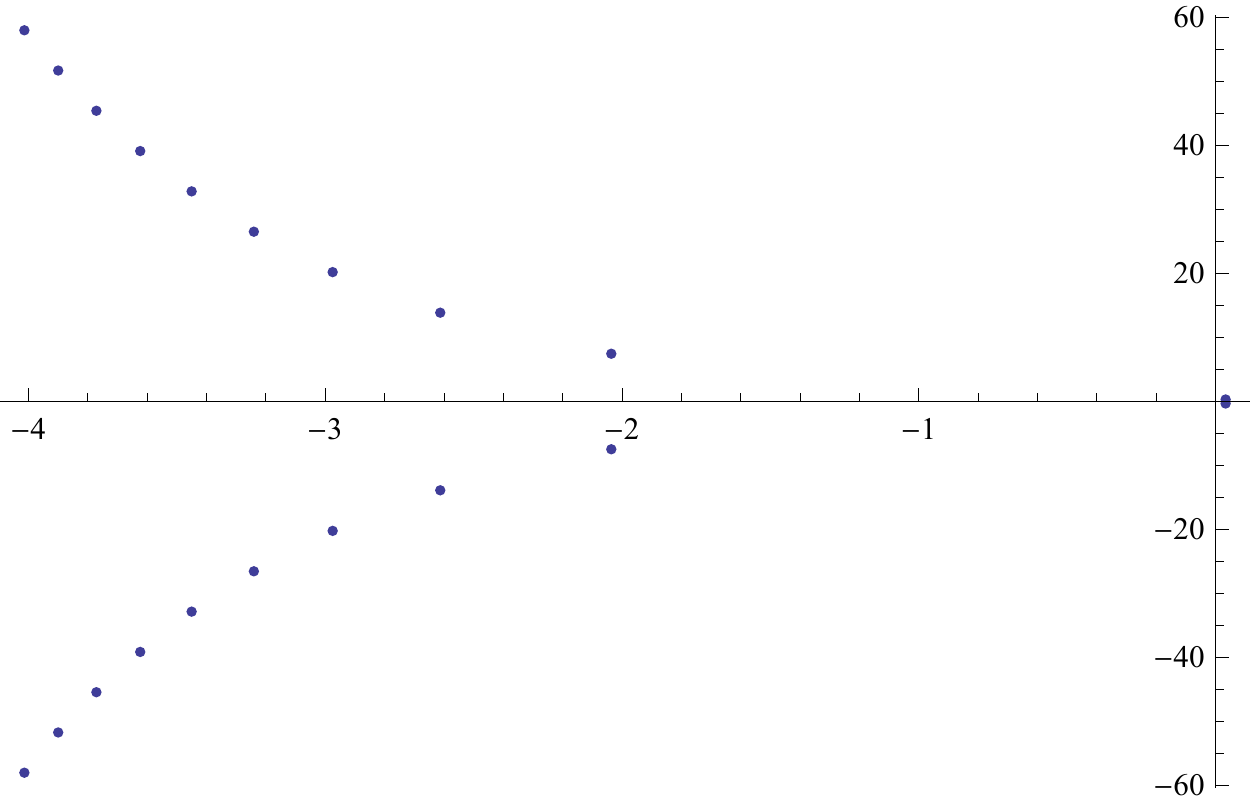}
\end{center}
\caption{Complex roots of \s{\tau(s)
	=1/\lambda - e^s(1-s)},
	shown for \s{\lambda=0.95}.}
\label{fig:roots-alpha}
\end{figure}

\begin{figure}[h]
\begin{center}
\includegraphics[width=0.5\textwidth]{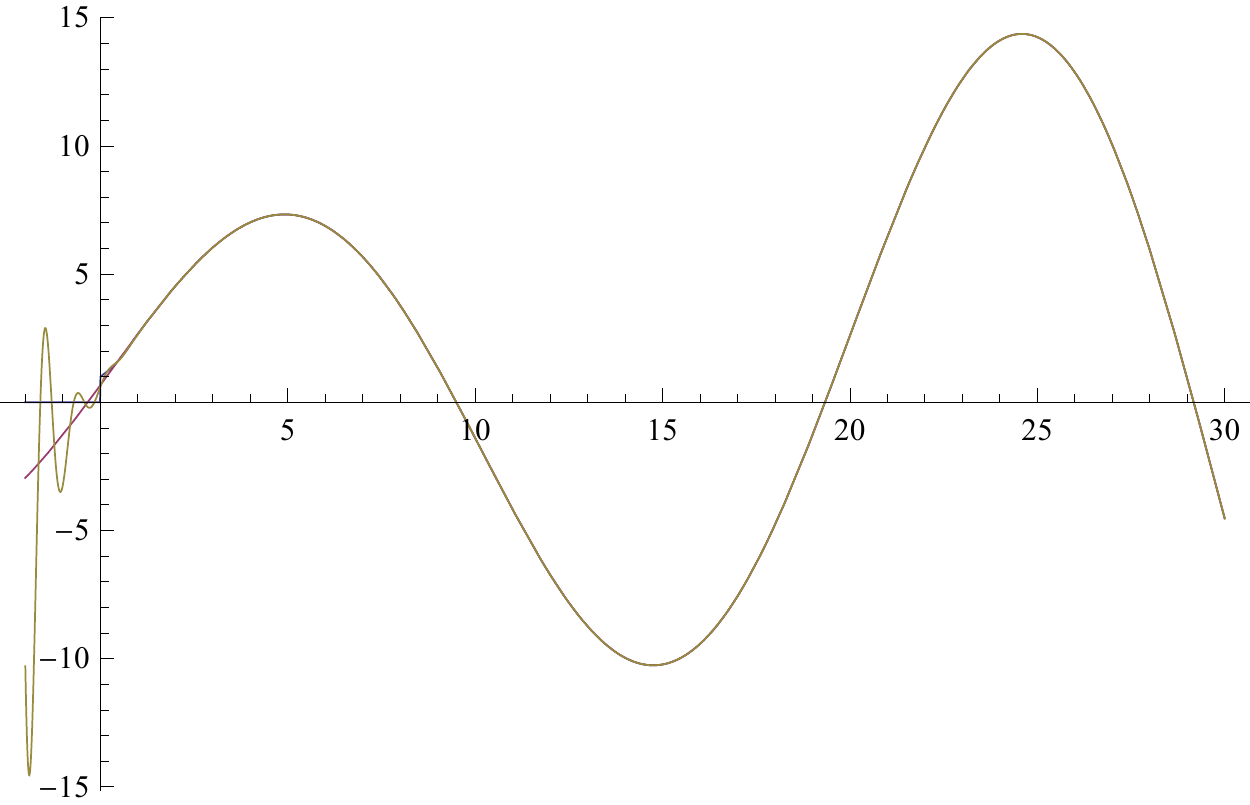}
\end{center}
\caption{Taking \s{\lm=0.95}, the figure shows
\s{g_{0.95}(x)} with its one-term and two-term approximations in the series expansion~\eqref{eq:g-series} (i.e., using the rightmost or two rightmost conjugate pairs of \s{s_k} from Fig.~\ref{fig:roots-alpha}).  The three curves are nearly indistinguishable
except for negative or very small~$x$.
}
\label{fig:series-x}
\end{figure}

\subsection{Asymptotics for \texorpdfstring{$\lambda$}{\textlambda} near 1}

We now extract the asymptotics of \s{g_\lm(x)} when \s{\lm\in(0,1]} is close to \s{1},
and \s{x} is large.
Recalling that \s{s_k} and \s{s_{-k-1}} are complex conjugates for \s{k\in\Z_{\ge0}}, we can re-express \eqref{e:g.lambda.sum.over.poles} as
	\beq\label{eq:g-series}
	g_\lm(x)
	= f_\lm(x) + 2\sum_{k\ge1} \real(e^{s_k x}/s_k),
	\quad
	f_\lm(x)
	= \left\{\hspace{-4pt}\begin{array}{ll}
	2\real(e^{s_0x}/s_0) & \text{if }\lm\in(0,1),\\
	2(x+1/3) & \text{if }\lm=1.
	\end{array}\right.
	\eeq
It is strongly suggested by \eqref{eq:s_k-asymptotic} and Fig.~\ref{fig:roots-alpha} that \s{f_\lm(x)} is a good approximation to \s{g_\lm(x)} in the limit of large positive \s{x}. We shall prove for \s{\lm} sufficiently near \s{1} that this is indeed true, which amounts to proving that \s{\real s_k} is strictly maximized over \s{k\in\Z_{\ge0}} at \s{k=0}.
Recalling \eqref{e:def.s.k.using.W.k},
it is clearly equivalent to prove
that \s{\real w_k} is strictly maximized over \s{k\in\Z_{\ge0}} at \s{k=0}.
For \s{\lm} near \s{1}, \s{w_0} lies near \s{W(-1/e)=-1}.
It is clear from the definition of the branch cuts
(\cite[\S4]{Lambert-W} and Fig.~\ref{f:lambert.w}) that
for \s{\lambda\in(0,1]} and \s{k\ge0},
	\beq\label{Im-W_k-range}
	\imag w_k \in [2k\pi,(2k+1)\pi)\,,
	\eeq
so \s{|w_k|\ge2\pi} for \s{k\ge1}. Combining with
\eqref{e:def.Lambert.W} gives
	\[
	\real w_k = -\log|e\lm|-\log|w_k|
	\le -1-\log(2\pi)-\log|\lm|
	\le -2
	\text{ for }
	|\lm|\text{ near }1\,.
	\]
Combining with \eqref{eq:s_k-asymptotic},
we see that there is a finite constant \s{C}
such that for all \s{\lm} in a neighborhood of \s{1},
\s{|g_\lm(x) - f_\lm(x)|\le C e^{-2x}}
for all \s{x>0}.

We next approximate \s{f_\lm(x)} for \s{\lm = e^{-\delta}}, \s{\delta} a small positive real. Taking \s{z=-1/(e\lm)} and \s{p=\sqrt{2(1+ez)}},
for \s{|p|<\sqrt{2}} there is a convergent series expansion \cite[eqn.~(4.22)]{Lambert-W}
	\[
	s_0=1+W_0(z)
	= \sum_{\ell\ge1}\mu_\ell p^\ell
	= p-\f13 \, p^2 + \f{11}{72} \, p^3 + O(p^4)\,.
	\]
Since \s{p= i\sqrt{2(e^\delta-1)} = i \sqrt{2\delta}[1+O(\delta)]}, the first term in the expansion matches what we have already noted in \eqref{eq:treefunctionatcritical}. Therefore
	\[
	\real s_0 = (2\delta/3)[1+O(\delta)]\,,
	\quad\quad
	\imag s_0 = \sqrt{2\delta}[1+O(\delta)]\,,
	\]
and as a result, for \s{\lm=e^{-\delta}}
with \s{\delta>0} small,
	\[
	\begin{array}{rl}
	f_\lm(x) 
	\hspace{-6pt}&= 2\exp\{(\real s_0)x\}
	\sin\Big[ (\imag s_0)x + \arctan[(\real s_0)/(\imag s_0)] \Big] \Big/ |s_0|
	\\
	\hspace{-6pt}&= 2\exp\{ (2\delta/3)x[1+O(\delta)]\}
	\sin\Big[
	[1+O(\delta)]
	\sqrt{2\delta}(x+1/3)
	\Big] \Big/ \sqrt{2\delta}\,.
	\end{array}
	\]
In particular, taking \s{\delta\downarrow0} with \s{x} fixed, we recover  \s{f_1(x) = 2(x+1/3)}.

Lastly we identify the value of \s{H} for which \s{g_\lm},
\s{\lm=e^{-\delta}}, gives the principal eigenfunction.
Recalling the discussion of \S\ref{ss:eig.intro}, \s{H+1}
corresponds to the smallest positive zero of \s{g_\lm}.
From the above, \s{f_\lm(x_0)=0} for
	\[
	x_0 \equiv \f{\pi-\arctan[(\real s_0)/(\imag s_0)]}
		{\imag s_0} 
	= \f{\pi}{\sqrt{2\delta}} - \f13 
	+ O(\sqrt{\delta})\,.
	\]
Recalling that \s{f_\lm} approximates \s{g_\lm}, we can write
	\[
	\begin{array}{rl}
	g_\lm(x_0+x)
	\hspace{-6pt}&= O(\exp\{-\Omega(x_0+x)\})
		+ f_\lm(x_0+x)\\
	\hspace{-6pt}&=O(\exp\{-\Omega(x_0+x)\})
		+ 2\exp\{(\real s_0)(x_0+x)\}
		\sin[ (\imag s_0) x ]\,.
	\end{array}
	\]
From this we see that we can choose \s{x_\star=x_0 + \exp\{-\Omega(x_0)\}} and a sufficiently large constant \s{x'} such that \s{g_\lm} has a root at \s{x_\star}, and is non-vanishing between \s{x'} and \s{x_\star}.
To rule out zeroes on \s{[0,x']}, note that \s{g_1} is non-zero on all of \s{\R_{\ge0}}, since no finite \s{H} has an eigenvalue of \s{1}. It follows by continuity in \s{g_\lm} that for \s{\lm} sufficiently near \s{1}, \s{g_\lm} is non-vanishing on \s{[0,x']}. Therefore
	\[ 
	H = x_\star-1
	= \f{\pi}{\sqrt{2\delta}} - \f43 
	+ O(\sqrt{\delta})\,, \quad\quad
	\delta=
	\f{\pi^2}{2}\f{1+O(1/H^2)}{(H+4/3)^2}\,.
	\]

\bibliographystyle{alphaabbr}
\bibliography{cycles}

\newcommand{\etalchar}[1]{$^{#1}$}
\begin{thebibliography}{GGTW09}

\bibitem[AFW12]{MR2927630}
O.~Angel, A.~D. Flaxman, and D.~B. Wilson.
\newblock A sharp threshold for minimum bounded-depth and bounded-diameter
  spanning trees and {S}teiner trees in random networks.
\newblock \href{http://dx.doi.org/10.1007/s00493-012-2552-z}{{\em
  Combinatorica}, 32(1):1--33}, 2012.

\bibitem[Ald01]{MR1839499}
D.~J. Aldous.
\newblock The {$\zeta(2)$} limit in the random assignment problem.
\newblock \href{http://dx.doi.org/10.1002/rsa.1015}{{\em Random Structures
  Algorithms}, 18(4):381--418}, 2001.

\bibitem[BGRS04]{MR2071332}
B.~Bollob{\'a}s, D.~Gamarnik, O.~Riordan, and B.~Sudakov.
\newblock On the value of a random minimum weight {S}teiner tree.
\newblock \href{http://dx.doi.org/10.1007/s00493-004-0013-z}{{\em
  Combinatorica}, 24(2):187--207}, 2004.

\bibitem[CFMS09]{MR2551020}
P.~Chebolu, A.~Frieze, P.~Melsted, and G.~B. Sorkin.
\newblock Average-class analyses of {V}ickrey costs.
\newblock In \href{http://dx.doi.org/10.1007/978-3-642-03685-9_33}{{\em
  Approximation, randomization, and combinatorial optimization}, Lecture Notes
  in Comput.\ Sci.\ \#5687, pages 434--447}. Springer, 2009.

\bibitem[CGH{\etalchar{+}}96]{Lambert-W}
R.~M. Corless, G.~H. Gonnet, D.~E.~G. Hare, D.~J. Jeffrey, and D.~E. Knuth.
\newblock On the {L}ambert {$W$} function.
\newblock \href{http://dx.doi.org/10.1007/BF02124750}{{\em Adv.\ Comput.\
  Math.}, 5(4):329--359}, 1996.

\bibitem[Das04]{dasdan}
A.~Dasdan.
\newblock Experimental analysis of the fastest optimum cycle ratio and mean
  algorithms.
\newblock \href{http://dx.doi.org/10.1145/1027084.1027085}{{\em ACM Trans.\
  Des.\ Autom.\ Electron.\ Syst.}, 9(4):385--418}, 2004.

\bibitem[DG98]{DG98}
A.~Dasdan and R.~K. Gupta.
\newblock Faster maximum and minimum mean cycle algorithms for
  system-performance analysis.
\newblock \href{http://dx.doi.org/10.1109/43.728912}{{\em IEEE Trans.\ on CAD
  of Integrated Circuits and Systems}, 17(10):889--899}, 1998.

\bibitem[DIG99]{DG99}
A.~Dasdan, S.~S. Irani, and R.~K. Gupta.
\newblock Efficient algorithms for optimum cycle mean and optimum cost to time
  ratio problems.
\newblock In \href{http://dx.doi.org/10.1145/309847.309862}{{\em Proceedings of
  the 36th Annual ACM/IEEE Design Automation Conference}, pages 37--42}, 1999.

\bibitem[Din13]{ding}
J.~Ding.
\newblock Scaling window for mean-field percolation of averages.
\newblock \href{http://dx.doi.org/10.1214/12-AOP765}{{\em Ann.\ Probab.},
  41(6):4407--4427}, 2013.

\bibitem[FG85]{MR770869}
A.~M. Frieze and G.~R. Grimmett.
\newblock The shortest-path problem for graphs with random arc-lengths.
\newblock \href{http://dx.doi.org/10.1016/0166-218X(85)90059-9}{{\em Discrete
  Appl. Math.}, 10(1):57--77}, 1985.

\bibitem[FM89]{MR1054012}
A.~M. Frieze and C.~J.~H. McDiarmid.
\newblock On random minimum length spanning trees.
\newblock \href{http://dx.doi.org/10.1007/BF02125348}{{\em Combinatorica},
  9(4):363--374}, 1989.

\bibitem[Fri85]{MR770868}
A.~M. Frieze.
\newblock On the value of a random minimum spanning tree problem.
\newblock \href{http://dx.doi.org/10.1016/0166-218X(85)90058-7}{{\em Discrete
  Appl. Math.}, 10(1):47--56}, 1985.

\bibitem[Fri04]{MR2104159}
A.~Frieze.
\newblock On random symmetric travelling salesman problems.
\newblock \href{http://dx.doi.org/10.1287/moor.1040.0105}{{\em Math. Oper.
  Res.}, 29(4):878--890}, 2004.

\bibitem[GGTW09]{georgiadis}
L.~Georgiadis, A.~V. Goldberg, R.~E. Tarjan, and R.~F. Werneck.
\newblock An experimental study of minimum mean cycle algorithms.
\newblock In I.~Finocchi and J.~Hershberger, editors,
  \href{http://dx.doi.org/10.1137/1.9781611972894.1}{{\em
  A{LENEX}09---{W}orkshop on {A}lgorithm {E}ngineering \& {E}xperiments}, pages
  1--13}. SIAM, 2009.

\bibitem[GT89]{goldberg-tarjan}
A.~V. Goldberg and R.~E. Tarjan.
\newblock Finding minimum-cost circulations by canceling negative cycles.
\newblock \href{http://dx.doi.org/10.1145/76359.76368}{{\em J. Assoc.\ Comput.\
  Mach.}, 36(4):873--886}, 1989.

\bibitem[HHVM07]{MR2309622}
R.~v.~d. Hofstad, G.~Hooghiemstra, and P.~Van~Mieghem.
\newblock The weight of the shortest path tree.
\newblock \href{http://dx.doi.org/10.1002/rsa.20141}{{\em Random Structures
  Algorithms}, 30(3):359--379}, 2007.

\bibitem[HVM08]{MR2433939}
G.~Hooghiemstra and P.~Van~Mieghem.
\newblock The weight and hopcount of the shortest path in the complete graph
  with exponential weights.
\newblock \href{http://dx.doi.org/10.1017/S0963548308009176}{{\em Combin.
  Probab. Comput.}, 17(4):537--548}, 2008.

\bibitem[Jan99]{MR1723648}
S.~Janson.
\newblock One, two and three times {$\log n/n$} for paths in a complete graph
  with random weights.
\newblock \href{http://dx.doi.org/10.1017/S0963548399003892}{{\em Combin.
  Probab. Comput.}, 8(4):347--361}, 1999.
\newblock Random graphs and combinatorial structures (Oberwolfach, 1997).

\bibitem[JK{\L}P93]{JKLP}
S.~Janson, D.~E. Knuth, T.~{\L}uczak, and B.~Pittel.
\newblock The birth of the giant component.
\newblock \href{http://dx.doi.org/10.1002/rsa.3240040303}{{\em Random
  Structures Algorithms}, 4(3):231--358}, 1993.
\newblock With an introduction by the editors.

\bibitem[Kac45]{kac}
M.~Kac.
\newblock Random walk in the presence of absorbing barriers.
\newblock {\em Ann.\ Math.\ Statistics}, 16:62--67, 1945.

\bibitem[Kle67]{klein}
M.~Klein.
\newblock A primal method for minimal cost flows with applications to the
  assignment and transportation problems.
\newblock \href{http://dx.doi.org/10.1287/mnsc.14.3.205}{{\em Management
  Science}, 14(3):205--220}, 1967.

\bibitem[KMT76]{KMT2}
J.~Koml{\'o}s, P.~Major, and G.~Tusn{\'a}dy.
\newblock An approximation of partial sums of independent {RV}'s, and the
  sample {DF}. {II}.
\newblock \href{http://dx.doi.org/10.1007/BF00532688}{{\em Z.
  Wahrscheinlichkeitstheorie und Verw.\ Gebiete}, 34(1):33--58}, 1976.

\bibitem[KO81]{KO81}
R.~M. Karp and J.~B. Orlin.
\newblock Parametric shortest path algorithms with an application to cyclic
  staffing.
\newblock \href{http://dx.doi.org/10.1016/0166-218X(81)90026-3}{{\em Discrete
  Appl.\ Math.}, 3(1):37--45}, 1981.

\bibitem[KW98]{kleinberg-williamson}
J.~Kleinberg and D.~Williamson, 1998.
\newblock Unpublished note.

\bibitem[LW04]{MR2036492}
S.~Linusson and J.~W{\"a}stlund.
\newblock A proof of {P}arisi's conjecture on the random assignment problem.
\newblock \href{http://dx.doi.org/10.1007/s00440-003-0308-9}{{\em Probab.
  Theory Related Fields}, 128(3):419--440}, 2004.

\bibitem[MP85]{mezard-parisi-replicas-opt}
M.~M\'ezard and G.~Parisi.
\newblock Replicas and optimization.
\newblock \href{http://dx.doi.org/10.1051/jphyslet:019850046017077100}{{\em J.
  Phys. Lett.}, 46(17):771--778}, 1985.

\bibitem[MP86]{mezard-parisi-tsp}
M.~M\'ezard and G.~Parisi.
\newblock A replica analysis of the travelling salesman problem.
\newblock \href{http://dx.doi.org/10.1051/jphys:019860047080128500}{{\em J.
  Phys. France}, 47(8):1285--1296}, 1986.

\bibitem[MW13]{mathieu-wilson}
C.~Mathieu and D.~B. Wilson.
\newblock The min mean-weight cycle in a random network.
\newblock \href{http://dx.doi.org/10.1017/S0963548313000229}{{\em Combin.\
  Probab.\ Comput.}, 22(5):763--782}, 2013.

\bibitem[NPS05]{MR2178256}
C.~Nair, B.~Prabhakar, and M.~Sharma.
\newblock Proofs of the {P}arisi and {C}oppersmith-{S}orkin random assignment
  conjectures.
\newblock \href{http://dx.doi.org/10.1002/rsa.20084}{{\em Random Structures
  Algorithms}, 27(4):413--444}, 2005.

\bibitem[OM00]{ouorou-mahey}
A.~Ouorou and P.~Mahey.
\newblock A minimum mean cycle cancelling method for nonlinear multicommodity
  flow problems.
\newblock \href{http://dx.doi.org/10.1016/S0377-2217(99)00050-8}{{\em European
  J. Oper.\ Res.}, 121(3):532--548}, 2000.

\bibitem[PP95]{pemantle-peres}
R.~Pemantle and Y.~Peres.
\newblock Critical random walk in random environment on trees.
\newblock {\em Ann.\ Probab.}, 23(1):105--140, 1995.

\bibitem[RG94]{radzik-goldberg}
T.~Radzik and A.~V. Goldberg.
\newblock Tight bounds on the number of minimum-mean cycle cancellations and
  related results.
\newblock \href{http://dx.doi.org/10.1007/BF01240734}{{\em Algorithmica},
  11(3):226--242}, 1994.

\bibitem[Sta99]{stanley}
R.~P. Stanley.
\newblock {\em \href{http://www-math.mit.edu/~rstan/ec/}{Enumerative
  Combinatorics}, {V}ol.~2}.
\newblock Cambridge studies in advanced mathematics \#62. Cambridge Univ.\
  Press, 1999.
\newblock With a foreword by G.-C. Rota and an appendix by S. Fomin.

\bibitem[W{\"a}s10]{MR2600434}
J.~W{\"a}stlund.
\newblock The mean field traveling salesman and related problems.
\newblock \href{http://dx.doi.org/10.1007/s11511-010-0046-7}{{\em Acta Math.},
  204(1):91--150}, 2010.

\bibitem[Wri49]{wright}
E.~M. Wright.
\newblock The linear difference-differential equation with constant
  coefficients.
\newblock \href{http://dx.doi.org/10.1017/S0080454100006804}{{\em Proc.\ Roy.\
  Soc.\ Edinburgh. Sect.~A.}, 62:387--393}, 1949.

\bibitem[YTO91]{YTO}
N.~E. Young, R.~E. Tarjan, and J.~B. Orlin.
\newblock Faster parametric shortest path and minimum-balance algorithms.
\newblock \href{http://dx.doi.org/10.1002/net.3230210206}{{\em Networks},
  21(2):205--221}, 1991.

\end{thebibliography}

\end{document}